\documentclass[a4paper,11pt,reqno]{amsart}

\usepackage[margin=1in,includehead,includefoot]{geometry}
\usepackage[utf8]{inputenc}
\usepackage[T1]{fontenc}
\usepackage{lmodern,microtype}
\usepackage{hyperref}
\usepackage{url}
\usepackage{booktabs}
\usepackage{mathtools,amssymb,amsthm}
\usepackage[foot]{amsaddr}
\usepackage[b]{esvect} 
\usepackage{graphicx, color}
\usepackage{wrapfig}
\usepackage[margin=5mm]{caption}
\usepackage{enumitem}
\usepackage[textsize=footnotesize,disable]{todonotes}
\usepackage{mdframed}
\usepackage{xpatch}
\usepackage{xcolor}
\usepackage{tikz}
\usetikzlibrary{arrows,positioning,shadows,shadings}
\usepackage{mycommands}
\usepackage{verbatim}

\graphicspath{{./graphics/}}

\newtheorem{theorem}{Theorem}
\newtheorem{lemma}[theorem]{Lemma}
\newtheorem{corollary}[theorem]{Corollary}
\newtheorem{open}{Problem}

\hypersetup{
  pdftitle={Combinatorial generation via permutation languages. V. Acyclic orientations},
  pdfauthor={Jean Cardinal, Hung P. Hoang, Arturo Merino, Ondrej Micka, Torsten M\"utze}
}

\title{Combinatorial generation via permutation languages. \\ V. Acyclic orientations}

\author{Jean Cardinal}
\address[Jean Cardinal]{Computer Science Department, Universit\'e Libre de Bruxelles, Belgium}
\email{jean.cardinal@ulb.be}

\author{Hung P. Hoang}
\address[Hung P. Hoang]{Department of Computer Science, ETH Z\"urich, Switzerland}
\email{hung.hoang@inf.ethz.ch}

\author{Arturo Merino}
\address[Arturo Merino]{Department of Mathematics, TU Berlin, Germany}
\email{merino@math.tu-berlin.de}

\author{Ond\v{r}ej Mi\v{c}ka}
\address[Ond\v{r}ej Mi\v{c}ka]{Department of Theoretical Computer Science and Mathematical Logic, Charles University, Prague, Czech Republic}
\email{micka@ktiml.mff.cuni.cz}

\author{Torsten M\"utze}
\address[Torsten M\"utze]{Department of Computer Science, University of Warwick, United Kingdom \& Department of Theoretical Computer Science and Mathematical Logic, Charles University, Prague, Czech Republic}
\email{torsten.mutze@warwick.ac.uk}

\thanks{An extended abstract of this paper was accepted for presentation at SODA~2023.
Arturo Merino was supported by ANID Becas Chile 2019-72200522.
Torsten M\"utze was supported by Czech Science Foundation grant GA~22-15272S}

\begin{document}

\begin{abstract}
In 1993, Savage, Squire, and West described an inductive construction for generating every acyclic orientation of a chordal graph exactly once, flipping one arc at a time.
We provide two generalizations of this result.
Firstly, we describe Gray codes for acyclic orientations of hypergraphs that satisfy a simple ordering condition, which generalizes the notion of perfect elimination order of graphs.
This unifies the Savage-Squire-West construction with a recent algorithm for generating elimination trees of chordal graphs.
Secondly, we consider quotients of lattices of acyclic orientations of chordal graphs, and we provide a Gray code for them, addressing a question raised by Pilaud.
This also generalizes a recent algorithm for generating lattice congruences of the weak order on the symmetric group.
Our algorithms are derived from the Hartung-Hoang-M\"utze-Williams combinatorial generation framework, and they yield simple algorithms for computing Hamilton paths and cycles on large classes of polytopes, including chordal nestohedra and quotientopes.
In particular, we derive an efficient implementation of the Savage-Squire-West construction.
Along the way, we give an overview of old and recent results about the polyhedral and order-theoretic aspects of acyclic orientations of graphs and hypergraphs.
\end{abstract}

\maketitle

\section{Introduction}

In 1953, Frank Gray registered a patent~\cite{gray_1953} for a method to list all binary words of length~$n$ in such a way that any two consecutive words differ in exactly one bit, and he called it the \emph{binary reflected code}.
More generally, a \emph{combinatorial Gray code}~\cite{ruskey_2016} is a listing of all objects of a combinatorial class such that any two consecutive objects differ by a `small local change', sometimes also called a `flip'.
Over the years, Gray codes have been designed for numerous classes of combinatorial objects, including permutations, combinations, integer and set partitions, Catalan objects (binary trees, triangulations etc.), linear extensions of a poset, spanning trees or matchings of a graph etc.; see the surveys~\cite{MR1491049,muetze_2022}.
This area has been the subject of intensive research combining ideas from combinatorics, algorithms, graph theory, order theory, algebra, and discrete geometry.
This enabled recent exciting progress on long-standing problems in this area~(see e.g.~\cite{MR3775825}), and the development of versatile general techniques for designing Gray codes~\cite{MR3126386,MR2844089,MR4391718}.
One of the main applications of Gray codes is to efficiently generate a class of combinatorial objects~(see e.g.~\cite{MR2807540}), and many such algorithms are described in the most recent volume of Knuth's book `The Art of Computer Programming'~\cite{MR3444818}.

\subsection{The Steinhaus-Johnson-Trotter algorithm}
\label{sec:sjt}

The \emph{Steinhaus-Johnson-Trotter algorithm}, also known as `plain changes', is one of the classical Gray codes for generating permutations.
Specifically, it lists all permutations of $[n]:=\{1,2,\ldots,n\}$ so that every pair of successive permutations differs by exactly one adjacent transposition, i.e., by swapping two neighboring entries of the permutation.
Using suitable auxiliary arrays, this algorithm can be implemented in time~$\cO(1)$ per visited permutation.

The Steinhaus-Johnson-Trotter ordering of permutations can be defined inductively as follows:
For $n=1$ the listing consists only of a single permutation~1.
To construct the listing for permutations of~$[n]$ for $n\geq 2$, we consider the listing for permutations of~$[n-1]$, and we replace every permutation~$\pi$ in it by the sequence of permutations obtained by inserting the new largest symbol~$n$ in all possible positions in~$\pi$ from right to left, or from left to right, alternatingly.
It is easy to check that this indeed gives a listing of all permutations of~$[n]$ by adjacent transpositions.
Moreover, as $n!$ is even for $n\geq 2$, the listing is cyclic, i.e., the last and first permutation differ only in an adjacent transposition.
For example, for $n=2$ we get the listing $1{\color{red}2},{\color{red}2}1$, for $n=3$ we get $12{\color{red}3},1{\color{red}3}2,{\color{red}3}12,{\color{red}3}21,2{\color{red}3}1,21{\color{red}3}$, and for $n=4$ we get $123{\color{red}4},12{\color{red}4}3,1{\color{red}4}23,{\color{red}4}123,{\color{red}4}132,1{\color{red}4}32,13{\color{red}4}2,132{\color{red}4},312{\color{red}4},31{\color{red}4}2,3{\color{red}4}12,{\color{red}4}312,{\color{red}4}321,3{\color{red}4}21,32{\color{red}4}1,321{\color{red}4},231{\color{red}4},\linebreak[4] 23{\color{red}4}1,2{\color{red}4}31,{\color{red}4}231,{\color{red}4}213,2{\color{red}4}13,21{\color{red}4}3,213{\color{red}4}$; see Figure~\ref{fig:p4sjt}.
In those listings, the newly inserted symbol~$n$ is highlighted, which allows tracking its zigzag movement.

Williams~\cite{MR3126386} found a strikingly simple equivalent description of the Steinhaus-Johnson-Trotter ordering via the following greedy algorithm: Start with the identity permutation, and repeatedly perform an adjacent transposition with the largest possible value that yields a previously unvisited permutation.

The results in this paper can be seen as far-ranging generalizations of these two alternative descriptions of the same fundamental ordering.

\subsection{Flip graphs, lattices, and polytopes}
\label{sec:flip}

Any Gray code problem gives rise to a corresponding \emph{flip graph}, which has as vertices the combinatorial objects of interest, and an edge between any two objects that differ by the specified flip operation.

\begin{wrapfigure}{r}{0.5\textwidth}
\centering
\includegraphics[page=1]{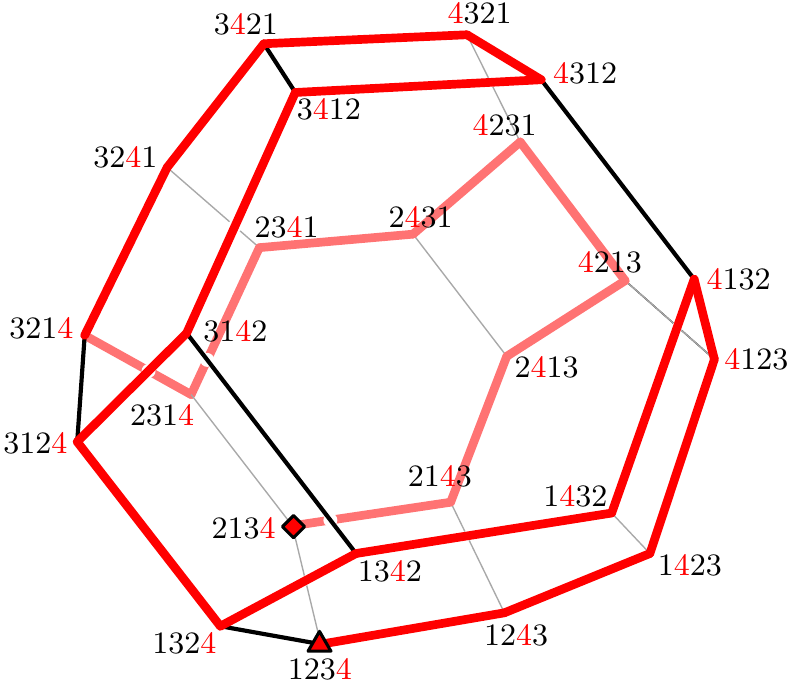}
\caption{The 3-dimensional permutohedron with the Steinhaus-Johnson-Trotter Hamilton path.
The start and end vertex are highlighted by a triangle and diamond, respectively, and can be joined to a Hamilton cycle.}
\label{fig:p4sjt}
\vspace{-3mm}
\end{wrapfigure}
For example, the flip graph on binary words of length~$n$ under flips that change a single bit is the $n$-dimensional hypercube.
Moreover, the flip graph for permutations under adjacent transpositions discussed in the previous section is the Cayley graph of the symmetric group generated by adjacent transpositions.
Another heavily studied example is the flip graph on binary trees under tree rotations~\cite{MR928904,MR3197650}.

Clearly, computing a Gray code for a set of combinatorial objects amounts to traversing a Hamilton path in the corresponding flip graph.
In particular, Hamilton paths in the three aforementioned flip graphs can be computed by the binary reflected code, the Steinhaus-Johnson-Trotter algorithm, and by an algorithm due to Lucas, Roelants van Baronaigien, and Ruskey~\cite{MR1239499}, respectively.

It turns out that many flip graphs can be equipped with a poset structure and realized as polytopes, i.e., they are cover graphs of certain lattices, and 1-skeleta of certain high-dimensional polytopes.
For example, the $n$-dimensional hypercube is the cover graph of the Boolean lattice and the skeleton of the Cartesian product~$[0,1]^n$.
Similarly, the flip graph on permutations under adjacent transpositions is the cover graph of the weak order on the symmetric group, and the skeleton of the \emph{permutohedron}; see Figure~\ref{fig:p4sjt}.
Lastly, the flip graph on binary trees under rotations is the cover graph of the Tamari lattice and the skeleton of the \emph{associahedron}.

Generalizations of these lattices and polytopes and the associated combinatorial structures have been the subject of intensive research in algebraic and polyhedral combinatorics; see Figure~\ref{fig:results}.
The theory of \emph{generalized permutohedra}~\cite{MR2520477,MR2487491,aguiar_ardila_2017} and of \emph{lattice congruences} and their \emph{quotients}~\cite{MR3221544,MR3645056}, in particular, provides us with a rich framework that contains all previous three examples as very special cases of a much broader picture.
Specifically, in the next three sections we discuss the generalizations shown one level above the bottom in Figure~\ref{fig:results}, namely acyclic orientations, elimination trees and lattice quotients.

\subsection{From permutations to acyclic orientations}
\label{sec:acyclic}

The starting point of this work are Gray codes for acyclic orientations of a graph.
Given a simple graph~$G$, an \emph{acyclic orientation} of~$G$ is a digraph~$D$ obtained by orienting every edge of~$G$ in one of two ways so that $D$ does not contain any directed cycles.
The goal is to list all acyclic orientations of~$G$ in such a way that any two consecutive orientations differ by reorienting a single arc, which we refer to as an \emph{arc flip}.

\begin{figure}[h!]
\centering
\includegraphics[page=1]{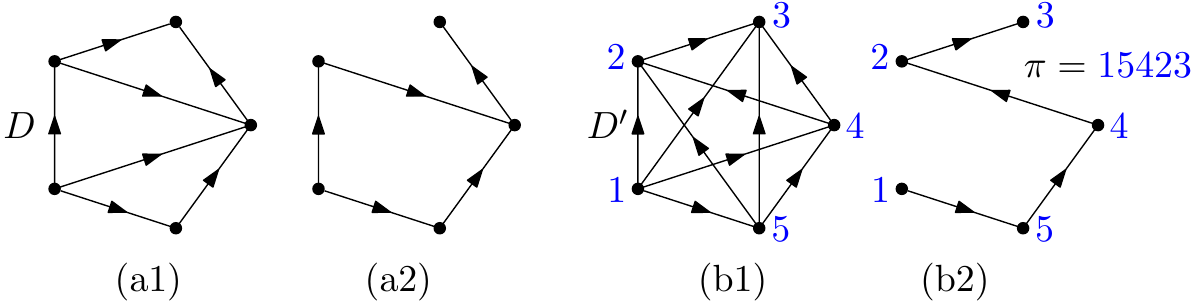}
\caption{(a1) An acyclic orientation~$D$ of a graph; (a2) the transitive reduction of~$D$, which contains precisely the flippable arcs in~$D$; (b1) an acyclic orientation~$D'$ of the complete graph; (b2) the transitive reduction of~$D'$ and the corresponding permutation.
}
\label{fig:acyclic}
\end{figure}

It is easy to see that in an acyclic orientation~$D$, the flippable arcs are precisely the arcs that are in the \emph{transitive reduction} of~$D$, which is the minimum subset of arcs that has the same reachability relations (i.e., the same transitive closure) as~$D$; see Figure~\ref{fig:acyclic}~(a1)+(a2).
If $G$ is the complete graph with vertex set~$[n]$, then the transitive reduction of any of its acyclic orientations~$D$ is a path, directed from the source to the sink of the orientation.
Consequently, we can interpret the vertex labels along this path as a permutation of~$[n]$; see Figure~\ref{fig:acyclic}~(b1)+(b2).
Furthermore, an arc flip corresponds to an adjacent transposition in this permutation.
Consequently, the flip graph on acyclic orientations of the complete graph is the skeleton of the permutohedron.
In general, the flip graph on the acyclic orientations of a graph~$G$ is the skeleton of a polytope known as the \emph{graphical zonotope} of~$G$~\cite{G77,MR712251,MR2383131}.

\begin{figure}[t!]
\centering
\includegraphics[page=2]{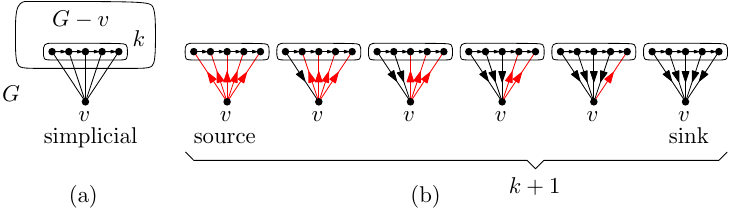}
\caption{Illustration of the Savage-Squire-West proof.
In the neighborhood of~$v$, only the transitive reduction is shown, whereas transitive arcs are omitted for simplicity.}
\label{fig:acyclic-flip}
\end{figure}

In general, not all (skeletons of) graphical zonotopes admit a Hamilton path or cycle, and we do not know of any simple conditions on the graph~$G$ for this to hold.
Clearly, the flip graph on acyclic orientations is bipartite for any graph~$G$, and if the partition classes have sizes that differ by more than~1, then this rules out the existence of a Hamilton path, a phenomenon that occurs for example if $G$ is a wheel graph with an even number of spokes~\cite{MR1267311}.
In this context, let us mention that counting the number of acyclic orientations of a graph is {\#}P-complete~\cite{MR830652}.

On the positive side, Savage, Squire and West~\cite{MR1267311} showed that the flip graph on acyclic orientations of~$G$ has a Hamilton cycle if~$G$ is chordal.
Their proof is a straightforward generalization of the Steinhaus-Johnson-Trotter construction, so we describe it here, with the goal of generalizing it even further subsequently.
A graph is \emph{chordal} if every induced cycle has length~3.
It is well-known that every chordal graph~$G$ has a \emph{simplicial} vertex~$v$, i.e., a vertex whose neighborhood in the graph is a clique.
We remove~$v$ from the graph, and by induction we obtain a Gray code for the acyclic orientations of~$G-v$; see Figure~\ref{fig:acyclic-flip}~(a).
Let $k$ be the number of neighbors of~$v$ in~$G$.
To construct the listing of acyclic orientations of~$G$, we replace every acyclic orientation in the listing for~$G-v$ by the sequence of $k+1$ acyclic orientations obtained by adding~$v$ and orienting the edges incident with~$v$ in all possible ways (that yield an acyclic orientation).
Specifically, since the neighborhood of~$v$ is a clique, whose transitive reduction is a path, there are precisely $k+1$ valid acyclic orientations obtained by adding~$v$, and they differ in a sequence of arc flips of arcs incident with~$v$, and this sequence starts and ends with~$v$ being a sink or a source; see Figure~\ref{fig:acyclic-flip}~(b).
In the Gray code for the acyclic orientations of~$G$, the vertex~$v$ alternates or `zigzags' between being sink or source.
As the number of acyclic orientations of any graph~$G$ with at least one edge is even (consider the involution on the set of all acyclic orientations that reorients every arc), the resulting ordering is cyclic, i.e., the last and first acyclic orientation differ only in an arc flip.
For $G$ being a complete graph, the resulting ordering of acyclic orientations and their corresponding permutations is exactly the Steinhaus-Johnson-Trotter ordering.

\subsection{From permutations to elimination trees}
\label{sec:elim}

\begin{figure}[b!]
\centering
\includegraphics[page=1]{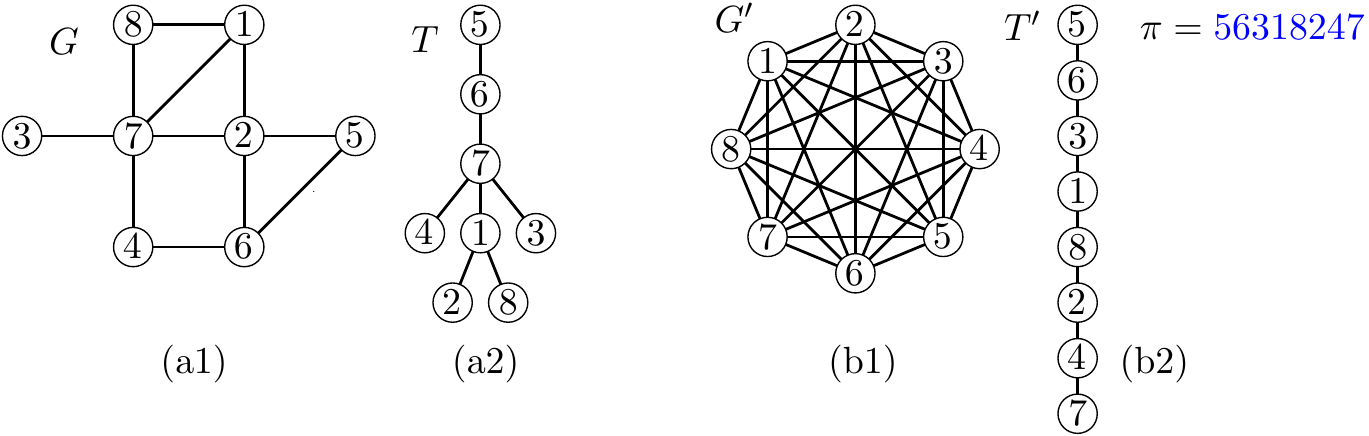}
\caption{(a1) A graph~$G$; (a2) an elimination tree~$T$ of~$G$; (b1) the complete graph; (b2) an elimination tree of the complete graph and the corresponding permutation.
}
\label{fig:elim}
\end{figure}

An elimination tree~$T$ of a connected graph~$G$ is an unordered rooted tree obtained as follows: We remove a vertex~$v$ of~$G$ which becomes the root of~$T$, and we recurse on the connected components of~$G-v$, whose elimination trees become the subtrees of~$v$ in~$T$; see Figure~\ref{fig:elim}~(a1)+(a2).
The goal is to list all elimination trees of~$G$ in such a way that any two consecutive trees differ by a \emph{rotation}, which is the result of swapping the removal order of two vertices that form a parent-child relationship in the elimination tree; see Figure~\ref{fig:hyper2}~(a)+(b).

Clearly, every elimination tree of a complete graph with vertex set~$[n]$ is a path, which can be interpreted as a permutation of~$[n]$ by reading the labels of the path from the root to the leaf; see Figure~\ref{fig:elim}~(b1)+(b2).
Furthermore, a tree rotation corresponds to an adjacent transposition in this permutation.
Consequently, the flip graph on elimination trees of the complete graph is the skeleton of the permutohedron.
In general, the flip graph on elimination trees of a graph~$G$ is the skeleton of a polytope known as the \emph{graph associahedron} of~$G$~\cite{MR2239078,MR2479448,MR2487491}.

\tikzset{
    double color fill/.code 2 args={
        \pgfdeclareverticalshading[%
            tikz@axis@top,tikz@axis@middle,tikz@axis@bottom%
        ]{diagonalfill}{100bp}{%
            color(0bp)=(tikz@axis@bottom);
            color(43.9bp)=(tikz@axis@bottom);
            color(44bp)=(tikz@axis@middle);
            color(44.1bp)=(tikz@axis@top);
            color(100bp)=(tikz@axis@top)
        }
        \tikzset{shade, left color=#1, right color=#2, shading=diagonalfill}
    }
}

\begin{figure}
\centering
\makebox[0cm]{ 
\begin{tikzpicture}[node distance=0.7cm and 0.3cm, auto]
\scriptsize
\tikzset{
oldnode/.style={rectangle,rounded corners,draw=black, fill=yellow!20, very thick, inner sep=1mm, minimum size=3em, text centered},
newnode/.style={rectangle,rounded corners,draw=black, fill=red!40, very thick, inner sep=1mm, minimum size=3em, text centered},
specialnode/.style={rectangle,rounded corners,draw=black, double color fill={yellow!20}{red!40}, shading angle=0,very thick, inner sep=1mm, minimum size=3em, text centered},
myarrow/.style={-, >=latex', shorten >=1pt, thick},
}

\node[newnode] (AOhyper) {
\begin{tabular}{l}
\hfill$\rightarrow$ Section~\ref{sec:res1} \\
C: \textbf{Acyclic orientations of hypergraphs} \\
\hspace{4mm}\textbf{in hyperfect elimination order} \\ \hline
P: Hypergraphic polytopes \\ \hline
H: New result: Theorem~\ref{thm:mainhyper} \\\hline
A: time $\cO(\Delta n)$, space $\cO(\Delta n^2)$ \\
\hspace{4mm}$\Delta=\Delta(\cH)$ max.\ degree, $n=|V|$ \\
\hspace{4mm}$\rightarrow$ Section~\ref{sec:algo}
\end{tabular}
};

\node[newnode, right=of AOhyper] (AOquotient) {
\begin{tabular}{l}
\hfill$\rightarrow$ Section~\ref{sec:res2} \\
L: \textbf{Quotients of acyclic reorientation} \\
\hspace{4mm}\textbf{lattices of peo-consistent digraphs} \\ \hline
P: Quotientopes \\ \hline
H: New result: Theorem~\ref{thm:mainquotient}
\end{tabular}
};

\node[newnode, below left=0.7cm and -3cm of AOhyper] (AObuilding) {
\begin{tabular}{l}
\hfill$\rightarrow$ Section~\ref{sec:res1} \\
C: \textbf{Acyclic orientations} \\
\hspace{4mm}\textbf{of chordal building sets} \\ \hline
P: Chordal nestohedra \\ \hline
H: New result: Theorem~\ref{thm:mainbuild}
\end{tabular}
};

\node[oldnode, below=of AObuilding] (ETgraph) {
\begin{tabular}{l}
\hfill$\rightarrow$ Section~\ref{sec:elim} \\
C: \textbf{Elimination trees} \\
\hspace{4mm}\textbf{of chordal graphs} \\ \hline
P: Chordal graph associahedra \\ \hline
H: \cite{MR3383157,DBLP:conf/soda/CardinalMM22} \\\hline
A: time $\cO(\sigma)$, space $\cO(n^2)$ \\
\hspace{4mm}$\sigma=\sigma(G)$ max.\ induced star, $n=|V|$
\end{tabular}
};

\node[specialnode, right=of ETgraph] (AOgraph) {
\begin{tabular}{l}
\hfill$\rightarrow$ Section~\ref{sec:acyclic} \\
C: \textbf{Acyclic orientations} \\
\hspace{4mm}\textbf{of chordal graphs} \\ \hline
P: Chordal graph zonotopes \\ \hline
H: \cite{MR1267311} \\\hline
A: time $\cO(\log \omega)$, space $\cO(n^2)$ \\
\hspace{4mm}$\omega=\omega(G)$ max.\ clique, $n=|V|$ \\
\hspace{4mm}$\rightarrow$ Section~\ref{sec:algo} and Theorem~\ref{thm:SSW-algo}
\end{tabular}
};

\node[oldnode, right=of AOgraph, yshift=-5.5mm] (WOquotient) {
\begin{tabular}{l}
\hfill$\rightarrow$ Section~\ref{sec:quotient} \\
L: \textbf{Quotients of} \\
\hspace{4mm}\textbf{the weak order} \\ \hline
P: Quotientopes \\ \hline
H: \cite{MR4344032}
\end{tabular}
};

\node[oldnode, below=1.55cm of ETgraph, xshift=-14mm] (Bin) {
\begin{tabular}{l}
\hfill$\rightarrow$ Section~\ref{sec:flip} \\
C: \textbf{Binary words} \\ \hline
L: Boolean lattice \\ \hline
P: Hypercube \\ \hline
H: Binary reflected code \\ \hline
A: time $\cO(1)$, space $\cO(n)$ \\
\hspace{4mm}$n=$ word length
\end{tabular}
};

\node[oldnode, below=1.4cm of AOgraph, xshift=-26.5mm, yshift=0mm] (WO) {
\begin{tabular}{l}
\hfill$\rightarrow$ Section~\ref{sec:flip} \\
C: \textbf{Permutations} \\ \hline
L: Weak order \\ \hline
P: Permutohedron \\ \hline
H: Steinhaus-Johnson-Trotter \\ \hline
A: time $\cO(1)$, space $\cO(n)$ \\
\hspace{4mm}$n=$ permutation length
\end{tabular}
};

\node[oldnode, below=1.4cm of WOquotient, xshift=-26.5mm] (BT) {
\begin{tabular}{l}
\hfill$\rightarrow$ Section~\ref{sec:flip} \\
C: \textbf{Binary trees} \\ \hline
L: Tamari lattice \\ \hline
P: Associahedron \\ \hline
H: \cite{MR1239499} \\ \hline
A: time $\cO(1)$, space $\cO(n)$ \\
\hspace{4mm}$n=$ number of nodes
\end{tabular}
};

\node[oldnode, right=of BT, xshift=-0.5mm] (Rect) {
\begin{tabular}{l}
\hfill$\rightarrow$ Section~\ref{sec:quotient} \\
C: \textbf{Diagonal / generic} \\
\hspace{4mm}\textbf{rectangulations} \\ \hline
L: $\mathrm{dRec}_n$ \cite{MR2871762} / $\mathrm{gRec}_n$ \cite{meehan_2019} \\ \hline
P: $P_\mathrm{dRec}$ \cite{MR2871762} / quotientope \\ \hline
H: \cite{perm_series_iii} \\ \hline
A: time $\cO(1)$, space $\cO(n)$ \\
\hspace{4mm}$n=$ number of rectangles
\end{tabular}
};

\draw[myarrow] (AOhyper.south) -- (AObuilding.north);
\draw[myarrow] (AOquotient.south) -- (AOgraph.north);
\draw[myarrow] (AObuilding.south) -- (ETgraph.north);
\draw[myarrow] (AOhyper.south) -- (AOgraph.north);
\draw[myarrow] (AOgraph.south) -- (WO.north);
\draw[myarrow] (ETgraph.south) -- (WO.north);
\draw[myarrow] (WOquotient.south) -- (WO.north);
\draw[myarrow] (ETgraph.south) -- (Bin.north);
\draw[myarrow] (Bin.north) -- (AOgraph.south);
\draw[myarrow] (Bin.north) -- (WOquotient.south);
\draw[myarrow] (AOquotient.south) -- (WOquotient.north);
\draw[myarrow] (BT.north) -- (ETgraph.south);
\draw[myarrow] (BT.north) -- (WOquotient.south);
\draw[myarrow] (Rect.north) -- (WOquotient.south);

\end{tikzpicture}
}
\medskip
\caption{Inclusion diagram of combinatorial families (C), lattices (L), polytopes (P), Hamiltonicity results (H), and corresponding algorithmic results (A).
More general objects are above their specialized counterparts.
New results are highlighted red.
Running times are per generated object, whereas the space refers to the total space needed (without storing previous objects).
Section references indicate where those results are discussed in more detail.
}
\label{fig:results}
\end{figure}
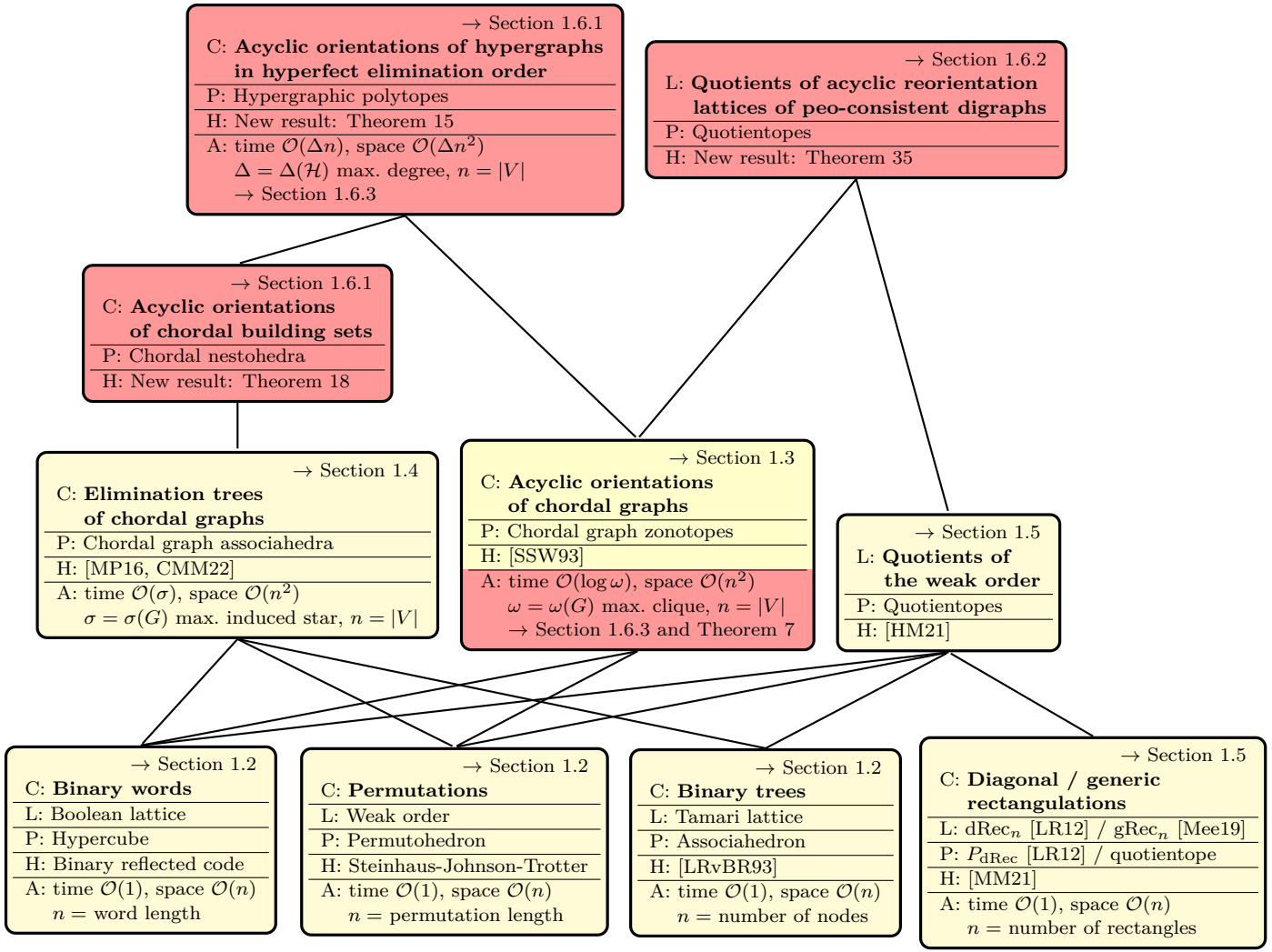

\todo{Replace references [CMM22], [Pil22] by journal references once available}

Manneville and Pilaud~\cite{MR3383157} showed that the skeleton of the graph associahedron of~$G$ admits a Hamilton cycle for any graph~$G$ with at least two edges.
In~\cite{DBLP:conf/soda/CardinalMM22}, we present a simple algorithm for computing a Hamilton path on the graph associahedron for the case when $G$ is a chordal graph.
This algorithm visits each elimination tree along the Hamilton path in time~$\cO(m+n)$, where $m$ and $n$ are the number of edges and vertices of~$G$, and this time can be improved to~$\cO(1)$ if $G$ is a tree.
Furthermore, if $G$ is chordal and 2-connected, then the resulting Hamilton path is actually a Hamilton cycle, i.e., the first and last elimination tree differ only in a tree rotation.
The proof in~\cite{DBLP:conf/soda/CardinalMM22} is an application of the Hartung-Hoang-M\"utze-Williams generation framework~\cite{MR4391718}, which generalizes the Steinhaus-Johnson-Trotter algorithm.
Specifically, we consider a simplicial vertex~$v$ in~$G$, we remove~$v$ from the graph, and by induction we obtain a rotation Gray code for the elimination trees of~$G-v$; see Figure~\ref{fig:elim-flip}~(a).
Let $N(v)$ be the set of neighbors of~$v$ in~$G$.
To construct the listing of elimination trees of~$G$, we consider every elimination tree~$T$ in the listing for~$G-v$.
As the vertices in~$N(v)$ form a clique in~$G$, these vertices appear on a path~$P$ in~$T$ that starts at the root and ends at a vertex from~$N(v)$.
We replace the elimination tree~$T$ for~$G-v$ by the sequence of elimination trees of~$G$ obtained by inserting~$v$ in all possibly ways on the path~$P$; see Figure~\ref{fig:elim-flip}~(b).
In particular, if $P$ has $k$ vertices, then we obtain $k+1$ elimination trees.
This insertion is done alternatingly from leaf to root or root to leaf, i.e., in the resulting listing of elimination trees of~$G$, the vertex~$v$ alternates or `zigzags' between being root or leaf.
In particular, for $G$ being a complete graph, the resulting ordering of elimination trees and their corresponding permutations is exactly the Steinhaus-Johnson-Trotter ordering.

\begin{figure}[h!]
\centering
\includegraphics[page=2]{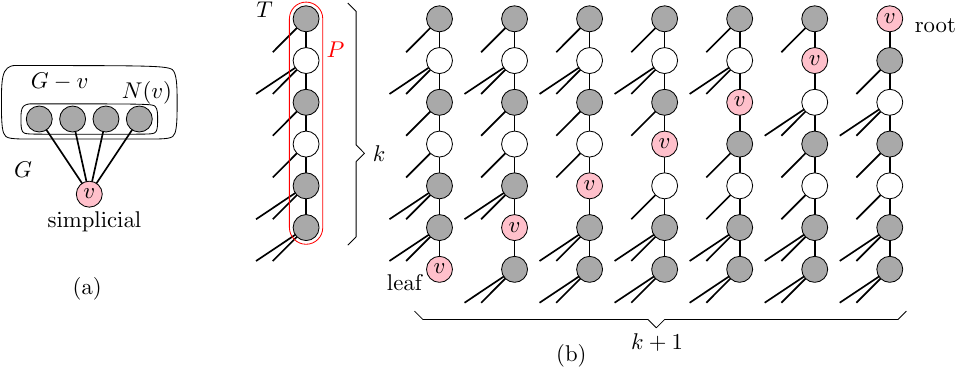}
\caption{Illustration of the zigzag argument for elimination trees of a chordal graph.
The vertices in the neighborhood~$N(v)$ of~$v$ are shaded, whereas other vertices of~$G$ are white.
The sloped edges in~$T$ connect to further subtrees (not shown).}
\label{fig:elim-flip}
\end{figure}

We will see that the appearance of chordal graphs in the aforementioned results on acyclic orientations and elimination trees is \textit{not a coincidence}.
In fact, our first main result gives a unified proof for both the results of~\cite{MR1267311} and~\cite{DBLP:conf/soda/CardinalMM22} by introducing a suitable notion of chordality for hypergraphs (see Section~\ref{sec:res1}).

\subsection{From permutations to lattice quotients}
\label{sec:quotient}

The \emph{inversion set} of a permutation~$\pi=a_1\cdots a_n$ is the set of pairs~$(a_i,a_j)$ that appear in the `wrong' order, i.e., the set $\{(a_i,a_j)\mid 1\leq i<j\leq n \text{ and } a_i>a_j\}$.
If we order all permutations of~$[n]$ by containment of their inversion sets, we obtain the \emph{weak order} on permutations; see Figure~\ref{fig:cong1}~(a).
The weak order forms a \emph{lattice}, i.e., joins and meets are well-defined.
Note that the cover relations in this lattice are precisely adjacent transpositions, i.e., the cover graph of this lattice is the skeleton of the permutohedron.
Furthermore, the levels $0,\ldots,\binom{n}{2}$ correspond to the number of inversions.

\begin{figure}[h!]
\centering
\includegraphics[page=1]{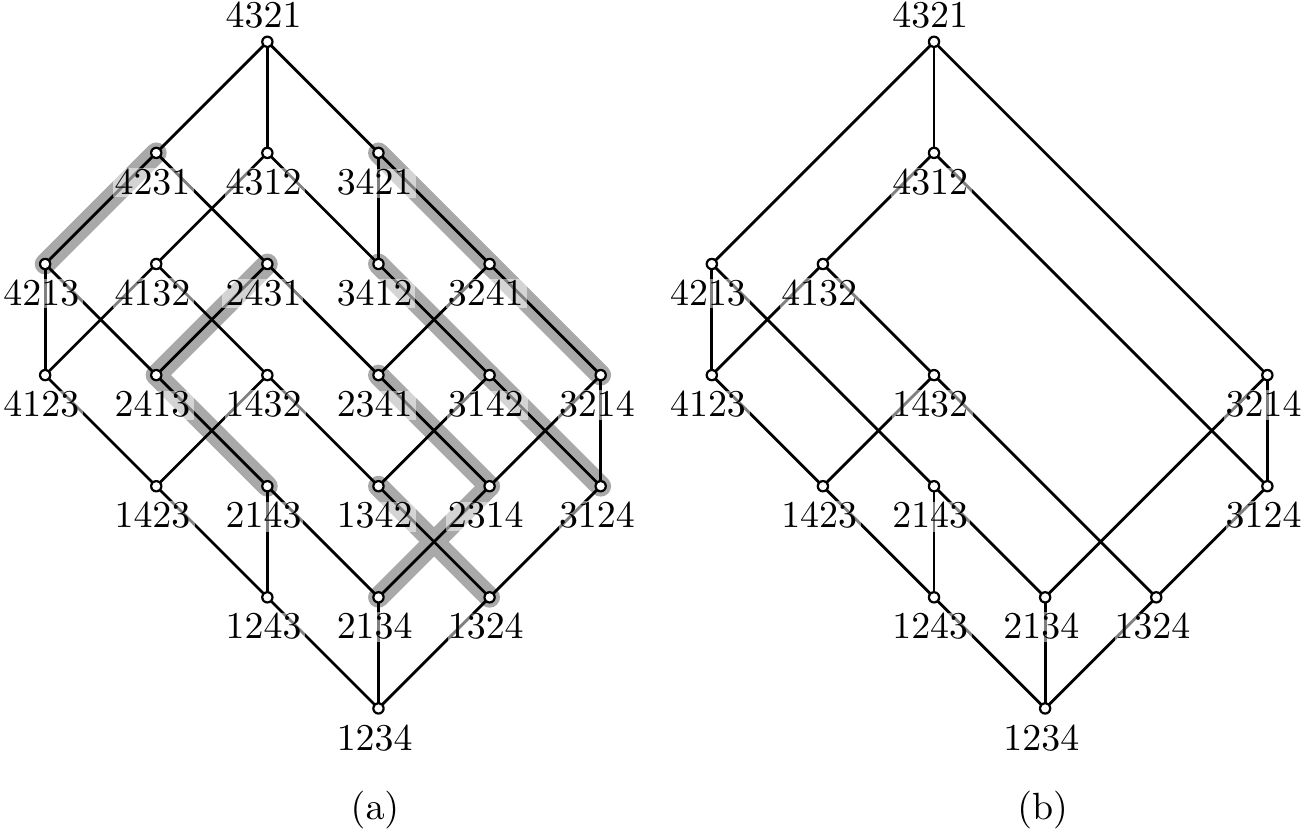}
\caption{(a) Weak order on permutations with the sylvester congruence, with the equivalence classes drawn as gray bubbles; (b) its quotient is the Tamari lattice of 231-avoiding permutations.}
\label{fig:cong1}
\end{figure}

A \emph{lattice congruence} is an equivalence relation on a lattice that respects joins and meets, i.e., for any choice of representatives from two equivalence classes, their joins and meets must lie in the same equivalence class.
A well-known example of a lattice congruence for the weak order is the \emph{sylvester congruence}, defined as the transitive closure of the rewriting rule $\_b\_ca\_ \equiv \_b\_ac\_$, where $a<b<c$, i.e., whenever a permutation contains three symbols $a<b<c$ in the order $b,c,a$, with $c$ and~$a$ directly next to each other, then this permutation belongs to the same equivalence class as the permutation obtained by transposing~$c$ and $a$.
Figure~\ref{fig:cong1}~(a) shows the equivalence classes of this congruence, with 231-avoiding permutations as the minima of the equivalence classes.
The \emph{quotient} of some lattice congruence is the lattice obtained by `contracting' each equivalence class to a single element; see Figure~\ref{fig:cong1}~(b).
In this way, we obtain for example the Tamari lattice (Figure~\ref{fig:cong1}~(b)) and the Boolean lattice as quotients of suitable lattice congruences of the weak order on permutations.
Let us also mention that the lattice of diagonal rectangulations~\cite{MR2871762} and the lattice of generic rectangulations~\cite{meehan_2019} arise as quotients of the weak order, and they have twisted Baxter permutations or 2-clumped permutations, respectively, as the minima of the equivalence classes.

The cover graphs of these quotient lattices are skeleta of polytopes known as \emph{quotientopes}~\cite{MR3964495,MR4311892}.
We showed in~\cite{MR4344032} that the skeleton of any quotientope admits a Hamilton path, and this Hamilton path can be computed by a `zigzag' strategy that generalizes the Steinhaus-Johnson-Trotter algorithm.

Pilaud~\cite{pilaud_2022} generalized this notion of quotientopes as follows:
He equipped the flip graph on acyclic orientations of a graph~$G$ with a poset structure.
Specifically, he considers the containment order of the sets of reoriented arcs with respect to some acyclic reference orientation~$D$ of~$G$.
The cover relations are given by reorienting a single arc, and the levels of this poset correspond to the number of reoriented arcs; see Figure~\ref{fig:cong2}~(a).
Pilaud characterized under which conditions on~$D$ this poset is a lattice, and he introduced lattice congruences and lattice quotients in this setting; see Figure~\ref{fig:cong2}~(b)+(c).
Furthermore, he showed how to realize the cover graphs of those quotients as polytopes, generalizing the constructions from~\cite{MR3964495,MR4311892}.
We saw before that if $G$ is a complete graph, then its acyclic orientations correspond to permutations, and arc flips correspond to adjacent transpositions, so in this special case Pilaud's lattice is precisely the weak order on permutations.
In his paper, Pilaud raised the problem which of these generalized quotientopes (parametrized by a reference orientation~$D$ of some graph~$G$) admit a Hamilton cycle.
The second main result of our work addresses Pilaud's question, by showing that they all have a Hamilton path, which can be computed by a simple greedy algorithm, again following the `zigzag' principle (see Section~\ref{sec:res2}).

\begin{figure}[h!]
\centering
\makebox[0cm]{ 
\includegraphics[page=2]{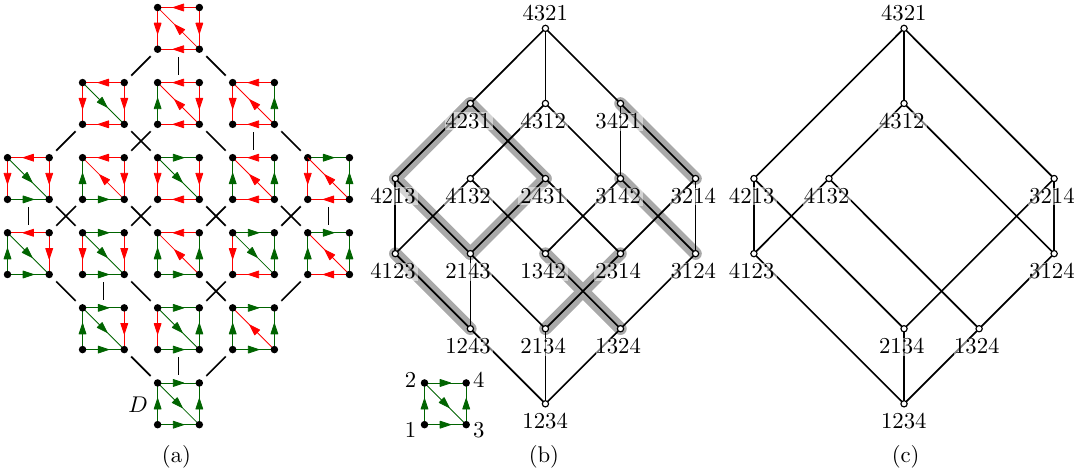}
}
\caption{(a) Lattice of acyclic reorientations of a digraph~$D$ (reoriented arcs w.r.t.~$D$ are highlighted); (b) one of its lattice congruences, and encoding of acyclic orientations by permutations; (c) the resulting quotient lattice, and corresponding representative permutations.}
\label{fig:cong2}
\end{figure}

\subsection{Our results}

We proceed to give an overview of the results of this paper, explaining the main statements and connections to previous work.
These new results are highlighted red in Figure~\ref{fig:results}.
A more formal treatment including proofs is provided in later sections.

\subsubsection{Acyclic orientations of hypergraphs}
\label{sec:res1}

Our first main contribution is the generalization of the Savage-Squire-West Gray code for acyclic orientations of chordal graphs to acyclic orientations of hypergraphs, by introducing a suitable notion of chordality for hypergraphs (see Theorem~\ref{thm:mainhyper} in Section~\ref{sec:aoh}).
Our construction yields Hamilton paths on the skeleta of certain \emph{hypergraphic polytopes}~\cite{MR3960512,aguiar_ardila_2017}, and in particular on chordal nestohedra~\cite{MR2520477} (Theorem~\ref{thm:mainbuild}).
Furthermore, this generalization subsumes the construction of Gray codes for elimination trees of chordal graphs presented in~\cite{DBLP:conf/soda/CardinalMM22} (Lemma~\ref{lem:BG-elim}).

Given a hypergraph $\cH=(V,\cE)$, where $\cE\seq 2^V$, an \emph{orientation} is a mapping $h:\cE\rightarrow V$ such that $h(A)\in A$ for every hyperedge~$A$ of~$\cH$; see Figure~\ref{fig:hyper}~(a).
The letter $h$ stands for `head': Every hyperedge designates one of its vertices as head.
This orientation is \emph{acyclic} if the digraph formed by all arcs $i\rightarrow j$ for every pair of distinct vertices $i,j\in V$ with $i,j\in A$ and $j=h(A)$ for some hyperedge~$A\in\cE$ is acyclic; see Figure~\ref{fig:hyper}~(b).
This definition clearly generalizes the notion of an acyclic digraph.
It is a special case of a more general definition recently used in a similar context by Benedetti, Bergeron, and Machacek~\cite{MR3960512}.

\begin{figure}[h!]
\centering
\includegraphics[page=1]{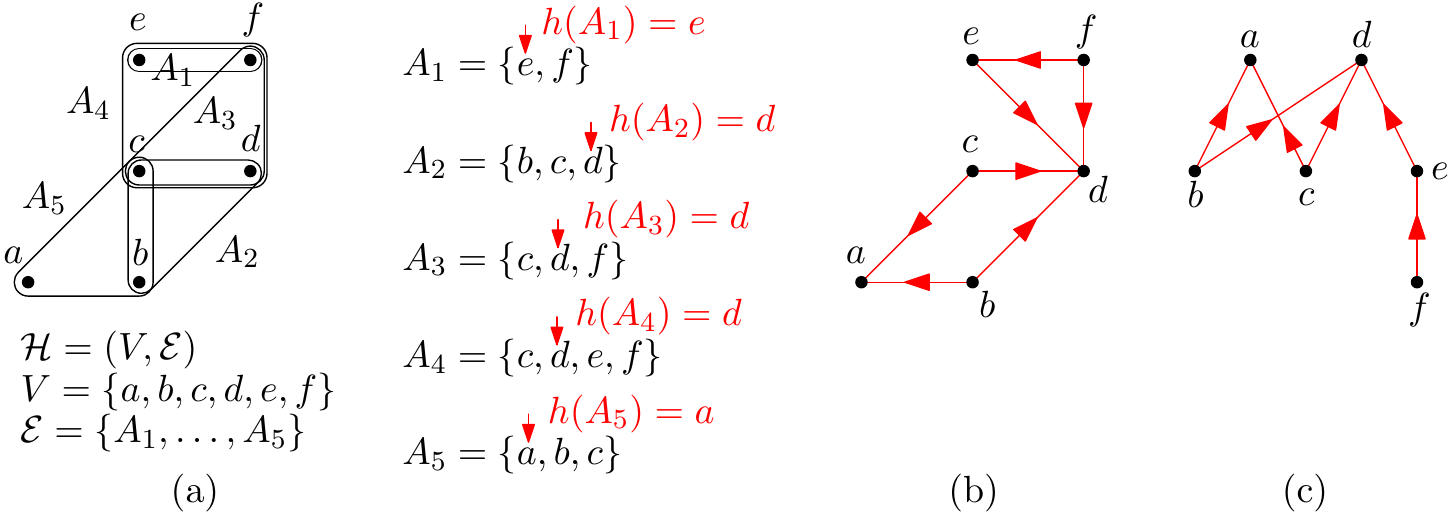}
\caption{(a) Acyclic orientation of a hypergraph; (b) corresponding acyclic digraph; (c) corresponding poset (whose cover graph is the transitive reduction of~(b)).}
\label{fig:hyper}
\end{figure}

Given a chordal graph~$G$, repeatedly removing one of its simplicial vertices yields a \emph{perfect elimination ordering (peo)} of the graph.
In fact, it is well-known that a graph~$G$ admits a perfect elimination ordering if and only if $G$ is chordal~\cite{MR186421}.
We generalize the notion of perfect elimination order of chordal graphs to what we call \emph{hyperfect elimination order} of hypergraphs (the formal definition is in Section~\ref{sec:hyperfect} below).
We then apply the aforementioned Hartung-Hoang-M\"utze-Williams generation framework to obtain a Gray code for the acyclic orientations of a hypergraph in hyperfect elimination order, using the `zigzag' idea common to Figures~\ref{fig:acyclic-flip} and~\ref{fig:elim-flip}.
The flip operation in an orientation~$h$ of a hypergraph $\cH=(V,\cE)$ consists of picking two vertices $i,j\in V$, and for all hyperedges $A\in\cE$ with $i,j\in A$ and $h(A)=j$ we instead define $h(A):=i$ (provided that the resulting orientation is acyclic); see Figure~\ref{fig:hyper2}~(c).

\begin{figure}[b!]
\centering
\includegraphics[page=2]{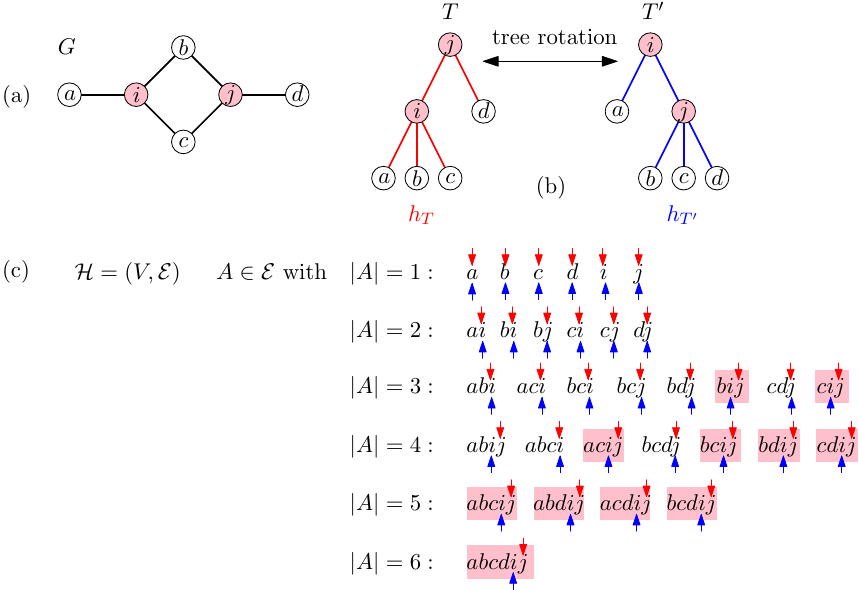}
\caption{(a) A graph~$G$; (b) two elimination trees~$T$ and~$T'$ of~$G$ that differ in a tree rotation; (c) corresponding two acyclic orientations~$h_T$ and~$h_{T'}$ of the graphical building set of~$G$, and flip operation between them.
The orientations $h_T$ and~$h_{T'}$ differ only on the highlighted hyperedges that contain both~$i$ and~$j$.
}
\label{fig:hyper2}
\end{figure}

The results from~\cite{DBLP:conf/soda/CardinalMM22} on elimination trees of chordal graphs can be recovered as a special case of our new results as follows; see Figure~\ref{fig:hyper2}:
Given a graph~$G=(V,E)$, its \emph{graphical building set} is the hypergraph $\cH=(V,\cE)$ such that
\begin{equation*}
\cE:=\{U\seq V\mid \text{$G[U]$ is connected}\},
\end{equation*}
where $G[U]$ is the subgraph of~$G$ induced by~$U$, i.e., we consider all connected subgraphs of~$G$ as hyperedges.
An elimination tree~$T$ of~$G$ with root~$v$ corresponds to the acyclic orientation of~$\cH$ in which every hyperedge~$A\in\cE$ with $v\in A$ satisfies $h(A)=v$, and this condition holds recursively for all subtrees.
Consequently, acyclic orientations of~$\cH$ are in one-to-one correspondence with elimination trees of~$G$, and the aforementioned flip operation on the hypergraph corresponds to a rotation in the elimination tree.
Our new Gray code on acyclic orientations of hypergraphs in hyperfect elimination order thus yields as a special case the Gray code on elimination trees of chordal graphs presented in~\cite{DBLP:conf/soda/CardinalMM22}.

The notions of building set and chordal building set have been defined abstractly without reference to a graph by Postnikov~\cite{MR2487491}, and Postnikov, Reiner, and Williams~\cite{MR2520477}, and the corresponding flip graphs arise as skeleta of \emph{chordal nestohedra}, a term also coined in \cite{MR2520477}.
Connections with acyclic orientations of hypergraphs have been investigated by Benedetti, Bergeron, and Machacek~\cite{MR3960512}.
We thus also obtain a simple constructive proof that chordal nestohedra admit a Hamilton path, directly yielding Gray codes for so-called \emph{$\cB$-forests}~\cite{MR2520477}.

\subsubsection{Quotients of acyclic reorientation lattices}
\label{sec:res2}

Recall from Section~\ref{sec:quotient} the definition of the poset of acyclic orientations of a graph~$G$ with respect to a reference orientation~$D$ of~$G$.
Pilaud~\cite{pilaud_2022} characterized when this poset is a lattice.
The following definitions are illustrated in Figure~\ref{fig:classes}.
Specifically, a digraph~$D$ is called \emph{vertebrate} if the transitive reduction of every induced subgraph of~$D$ is a forest.
It is easy to see that vertebrate implies acyclic.
Furthermore, $D$ is called \emph{filled} if for any directed path $v_1\rightarrow\cdots\rightarrow v_k$ in~$D$, if the arc~$v_1\rightarrow v_k$ belongs to~$D$, then all arcs $v_i\rightarrow v_j$, $1\leq i<j\leq k$, also belong to~$D$.
A digraph is called \emph{skeletal} if it is both vertebrate and filled.
Pilaud~\cite[Thm.~1+Thm.~3]{pilaud_2022} showed that the acyclic reorientation poset of~$D$ is a lattice if and only if $D$ is vertebrate, and that this lattice is semidistributive if and only if $D$ is filled.
He also raised the following question in his paper.

\begin{figure}[b!]
\centering
\includegraphics[page=3]{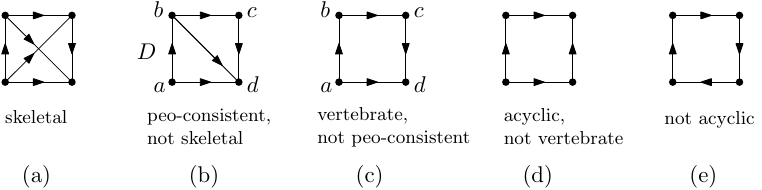}
\caption{Illustration of various classes of digraphs: (b) is peo-consistent, as the neighborhood of the source~$a$ is a clique, and $D-a$ is also peo-consistent; it is not skeletal, as the path $a\rightarrow b\rightarrow c\rightarrow d$ is not filled; (c) is vertebrate, as the transitive reduction is the path $a\rightarrow b\rightarrow c\rightarrow d$; it is not peo-consistent, as $a$ is a source and $d$ is a sink, but none of their neighborhoods is a clique; (d) is not vertebrate, as the transitive reduction is the full graph, which is not a forest; (e) is a directed cycle.}
\label{fig:classes}
\end{figure}

\begin{open}[{\cite[Problem~51]{pilaud_2022}}]
\label{prob:pilaud}
Given a skeletal (i.e., vertebrate and filled) digraph~$D$, do all cover graphs of lattice quotients of the acyclic reorientation lattice of~$D$ admit a Hamilton cycle?
\end{open}

Our second main contribution is to address Pilaud's question, by showing that those cover graphs all have a Hamilton path, which can be computed by a simple greedy algorithm, generalizing Steinhaus-Johnson-Trotter (Theorem~\ref{thm:mainquotient} in Section~\ref{sec:arl}).
This also yields an algorithmic proof that the corresponding generalized quotientopes admit a Hamilton path.
Furthermore, our result encompasses all earlier results on quotients of the weak order on permutations~\cite{MR4344032}, which are obtained as special case when $D$ is an acyclic orientation of a complete graph.

In fact, our results hold not only for skeletal digraphs~$D$, but for a slightly larger class.
Specifically, a \emph{peo-consistent} digraph $D$ has a source or sink~$v$ (i.e., all arcs incident with~$v$ are either outgoing or incoming, respectively) whose neighborhood is a clique, and $D-v$ is also peo-consistent or empty.
A straightforward induction shows that peo-consistent implies vertebrate.
The key fact we establish in our paper is that skeletal implies peo-consistent (Lemma~\ref{lem:skeletal}).
We thus have the following inclusions among classes of digraphs, which are strict (see Figure~\ref{fig:classes}):
\begin{equation}
\label{eq:inclusion}
\text{skeletal} \subset \text{peo-consistent} \subset \text{vertebrate} \subset \text{acyclic}.
\end{equation}

It is not difficult to see that an undirected graph has a peo-consistent orientation if and only if it is chordal.
We also observe that an undirected graph admits a skeletal orientation only if it is \emph{strongly chordal}~\cite{MR685625}.

We emphasize that our aforementioned results only guarantee a Hamilton path, whereas Pilaud's question asks for a Hamilton cycle.
However, our results hold for a slightly larger class of graphs, and they yield a simple algorithm.
As mentioned before, in the special case of elimination trees of chordal graphs~$G$, there are interesting cases where the Hamilton path computed by our algorithm is actually a Hamilton cycle, i.e., the first and last elimination tree differ only in a tree rotation.
Specifically, this happens if $G$ is chordal and 2-connected; see~\cite{DBLP:conf/soda/CardinalMM22}.
In particular, when $G$ is a complete graph our algorithm specializes to the Steinhaus-Johnson-Trotter algorithm for permutations, which produces a cyclic Gray code.

\subsubsection{Efficient generation algorithms}
\label{sec:algo}

We briefly discuss the computational efficiency of the Gray code algorithms derived from our work.
Those are summarized in the bottom part of the boxes in Figure~\ref{fig:results}.
All four algorithms mentioned at the bottom level in Figure~\ref{fig:results} can be implemented to output each new object in time~$\cO(1)$.

Our first algorithmic contribution is to turn the Savage-Squire-West Gray code for acyclic orientations of a chordal graph~$G$ into an algorithm that generates each acyclic orientation in time~$\cO(\log \omega)$ on average, where $\omega=\omega(G)$ is the clique number of~$G$ (Theorem~\ref{thm:SSW-algo} in Section~\ref{sec:aog}).
Clearly, we have $\omega\leq n$, which yields the more generous bound~$\cO(\log n)$.
The initialization time of our algorithm is~$\cO(n^2)$, which includes the time for testing chordality and computing a perfect elimination ordering, and the required space is~$\cO(n^2)$.
In our algorithm, we represent each acyclic orientation as an adjacency matrix, which allows constant-time orientation queries for each arc.
In the following we write $m$ and~$n$ for the number of edges and vertices of~$G$, respectively.
For comparison, Barbosa and Szwarcfiter~\cite{MR1733453} described an algorithm to generate each acyclic orientation of an arbitrary graph in time~$\cO(m+n)$ on average.
Moreover, Conte, Grossi, Marino, and Rizzi~\cite{MR3815526} provided an algorithm that generates each acyclic orientation in time~$\cO(m)$, and this bound holds in every iteration.
Their approach generalizes to the setting where some vertices are prescribed as sources, at the cost of a higher running time.
However, none of these other algorithms produces a Gray code listing, i.e., they do not yield Hamilton paths on the corresponding graphical zonotopes, unlike our Gray code.

Generalizing this algorithm, we can implement our Gray code for generating all acyclic orientations of a hypergraph~$\cH=(V,\cE)$ in hyperfect elimination order in time~$\cO(\Delta n)$ per generated acyclic orientation, where $\Delta=\Delta(\cH):=\max_{v\in V}|\{A\in \cE\mid v\in A\}|$ denotes the maximum degree and $n=|V|$ is the number of vertices.
The space required by this algorithm is~$\cO(\Delta n^2)$, and the initialization time is also~$\cO(\Delta n^2)$.
Testing whether a hypergraph admits a hyperfect elimination order and computing one takes time~$O(\Delta^3 n^5)$.

We implemented both of the aforementioned algorithms in C++, and we made this code available for download and experimentation on the Combinatorial Object Server~\cite{cos_orient}.
For algorithmic details on graphs, see Section~\ref{sec:SSW-algo}, and for details on hypergraphs, see our C++ implementation.

For lattice congruences and their quotients, it is difficult to provide meaningful statements about running times because of representation issues.
Specifically, for the weak order on permutations of~$[n]$, there are double-exponentially in~$n$ many different lattice congruences~\cite[Thm.~18]{MR4344032}.
Specifying the congruence as input of an algorithm therefore takes exponential space in general.
Consequently, one cannot improve much upon a naive specification of the congruence as a full list of equivalence classes as input of the algorithm.
However, given such a specification as input, there is no value in generating it efficiently.

\subsection{The permutation language framework}
\label{sec:framework}

In a recent line of work, Hartung, Hoang, M\"utze, and Williams~\cite{MR4391718} introduced a far-ranging generalization of the Steinhaus-Johnson-Trotter algorithm, which yields efficient Gray code algorithms for a large variety of combinatorial objects, based on encoding them as permutations.
So far, the framework has been applied successfully to obtain Gray codes for pattern-avoiding permutations~\cite{MR4391718}, lattice congruences of the weak order on permutations~\cite{MR4344032}, different families of rectangulations~\cite{perm_series_iii} (a rectangulation is a subdivision of a rectangle into smaller rectangles), and elimination trees of chordal graphs~\cite{DBLP:conf/soda/CardinalMM22}.
The methods described in this paper extend the reach of this framework, and make it applicable to generate even more general classes of objects, namely acyclic orientations of hypergraphs, which in particular subsumes the earlier results in~\cite{MR4344032} and~\cite{DBLP:conf/soda/CardinalMM22}.
In the following we summarize the key methods and results provided by this generation framework.

\subsubsection{Jumps in permutations}

We use $S_n$ to denote the set of all permutations of~$[n]$.
Furthermore, we use $\ide_n=12\cdots n$ to denote the identity permutation, and $\varepsilon\in S_0$ to denote the empty permutation.
A permutation~$\pi=a_1\cdots a_n$ is \emph{peak-free} if it does not contain any triple~$a_{i-1}<a_i>a_{i+1}$.
For any $\pi\in S_{n-1}$ and any $1\leq i\leq n$, we write $c_i(\pi)\in S_n$ for the permutation obtained from~$\pi$ by inserting the new largest value~$n$ at position~$i$ of~$\pi$, i.e., if $\pi=a_1\cdots a_{n-1}$ then $c_i(\pi)=a_1\cdots a_{i-1} \, n\, a_i \cdots a_{n-1}$.
Moreover, for~$\pi\in S_n$, we write $p(\pi)\in S_{n-1}$ for the permutation obtained from~$\pi$ by removing the largest entry~$n$.
Given a permutation $\pi=a_1\cdots a_n$ with a substring $a_i\cdots a_{i+d}$ with $d>0$ and $a_i>a_{i+1},\ldots,a_{i+d}$, a \emph{right jump of the value~$a_i$ by $d$~steps} is a cyclic left rotation of this substring by one position to $a_{i+1}\cdots a_{i+d} a_i$.
Similarly, given a substring $a_{i-d}\cdots a_i$ with $d>0$ and $a_i>a_{i-d},\ldots,a_{i-1}$, a \emph{left jump of the value~$a_i$ by $d$~steps} is a cyclic right rotation of this substring to $a_i a_{i-d}\cdots a_{i-1}$.
For example, a right jump of the value~5 in the permutation~$265134$ by 2 steps yields~$261354$.

\subsubsection{A simple greedy algorithm}
\label{sec:greedy}

The main ingredient of the framework is the following simple greedy algorithm to generate a set of permutations $L_n\seq S_n$.
We say that a jump is \emph{minimal} with respect to~$L_n$, if every jump of the same value in the same direction by fewer steps creates a permutation that is not in~$L_n$.

\begin{algo}{Algorithm~J}{Greedy minimal jumps}
This algorithm attempts to greedily generate a set of permutations $L_n\seq S_n$ using minimal jumps starting from an initial permutation $\pi_0 \in L_n$.
\begin{enumerate}[label={\bfseries J\arabic*.}, leftmargin=8mm, noitemsep, topsep=3pt plus 3pt]
\item{} [Initialize] Visit the initial permutation~$\pi_0$.
\item{} [Jump] Generate an unvisited permutation from~$L_n$ by performing a minimal jump of the largest possible value in the most recently visited permutation.
If no such jump exists, or the jump direction is ambiguous, then terminate.
Otherwise visit this permutation and repeat~J2.
\end{enumerate}
\end{algo}

Note that Algorithm~J is a generalization of Williams' greedy description of the Steinhaus-Johnson-Trotter algorithm given in Section~\ref{sec:sjt}.
Indeed, if $L_n=S_n$ is the set of all permutations of~$[n]$, then minimal jumps correspond to adjacent transpositions.

Note that by the definition of step~J2, Algorithm~J never visits any permutation twice.
The following key result provides a sufficient condition on the set~$L_n$ to guarantee that Algorithm~J succeeds to list all permutations from~$L_n$.
This condition is captured by the following closure property of the set~$L_n$.
A set of permutations~$L_n\seq S_n$ is called a \emph{zigzag language}, if either $n=0$ and $L_0=\{\varepsilon\}$, or if $n\geq 1$ and $L_{n-1}:=\{p(\pi)\mid \pi\in L_n\}$ is a zigzag language satisfying either one of the following conditions:
\begin{enumerate}[label={(z\arabic*)}, leftmargin=8mm, noitemsep, topsep=3pt plus 3pt]
\item For every $\pi\in L_{n-1}$ we have~$c_1(\pi)\in L_n$ and~$c_n(\pi)\in L_n$.
\item We have $L_n=\{c_n(\pi)\mid \pi\in L_{n-1}\}$.
\end{enumerate}

\begin{theorem}[\cite{MR4391718}]
\label{thm:jump}
Given any zigzag language of permutations~$L_n$ and initial permutation $\pi_0=\ide_n$, Algorithm~J visits every permutation from~$L_n$ exactly once.
\end{theorem}

It was already argued in~\cite{MR4391718} that more generally, any peak-free permutation can be used as initial permutation~$\pi_0$ for Algorithm~J.

We emphasize that Algorithm~J can be made \emph{history-free}, i.e., by introducing suitable auxiliary arrays, step~J2 can be performed without maintaining any previously visited permutations in order to decide which jump to perform; for details see \cite[Sec.~5.1+8.7]{perm_series_iii} and Section~\ref{sec:SSW-algo}.
The running time of this algorithm is then only determined by the time it takes to decide membership of a permutation in the zigzag language~$L_n$.
We should think of~$L_n$ as a set of permutations defined by some property, such as for example `permutations that avoid the pattern 231' or `permutations that encode acyclic orientations of some graph' (recall Figures~\ref{fig:cong1}~(b) and~\ref{fig:cong2}~(b), respectively), rather than an explicitly given set.
After all, if the set was already provided explicitly as input, then there would be no point in generating it; recall the discussion in Section~\ref{sec:algo}.

\subsubsection{Inductive description of the same ordering}
\label{sec:zigzag}

In the same way that the Steinhaus-Johnson-Trotter ordering can be defined both greedily or inductively, the ordering produced by Algorithm~J can also be defined inductively, as we show next.

Specifically, given a zigzag language~$L_n$, we write~$J(L_n)$ for the ordering of all permutations from~$L_n$ produced by Algorithm~J when initialized with~$\pi_0=\ide_n$.
For any $\pi\in L_{n-1}$ we let $\rvec{c}(\pi)$ be the sequence of all $c_i(\pi)\in L_n$ for $i=1,2,\ldots,n$, starting with $c_1(\pi)$ and ending with $c_n(\pi)$, and we let $\lvec{c}(\pi)$ denote the reverse sequence, i.e., it starts with $c_n(\pi)$ and ends with $c_1(\pi)$.
In words, those sequences are obtained by inserting into~$\pi$ the new largest value~$n$ from left to right, or from right to left, respectively, in all possible positions that yield a permutation from~$L_n$, skipping the positions that yield a permutation that is not in~$L_n$.
It was shown in~\cite{MR4391718} that the sequence~$J(L_n)$ can be described inductively as follows:
If $n=0$ then we have $J(L_0)=\varepsilon$, and if $n\geq 1$ then we consider the finite sequence $J(L_{n-1})=:\pi_1,\pi_2,\ldots$ and we have
\begin{subequations}
\label{eq:JLn12}
\begin{equation}
\label{eq:JLn1}
J(L_n)=\lvec{c}(\pi_1),\rvec{c}(\pi_2),\lvec{c}(\pi_3),\rvec{c}(\pi_4),\ldots
\end{equation}
if condition~(z1) holds, and
\begin{equation}
\label{eq:JLn2}
J(L_n)=c_n(\pi_1),c_n(\pi_2),c_n(\pi_3),c_n(\pi_4),\ldots
\end{equation}
\end{subequations}
if condition~(z2) holds.
In words, if condition~(z1) holds then this sequence is obtained from the previous sequence by inserting the new largest value~$n$ in all possible positions alternatingly from right to left, or from left to right, in a `zigzag' fashion.
The case where condition~(z2) holds is exceptional, as we only append~$n$ to each permutation of the previous sequence.

\subsubsection{How we apply Algorithm~J in this work}

To prove our results on acyclic orientations of hypergraphs discussed in Section~\ref{sec:res1}, we proceed as follows:
Given a hypergraph~$\cH$ in hyperfect elimination order, we label its vertices with $1,2,\ldots,n$ according to this ordering, and we encode any of its acyclic orientations as a permutation on~$[n]$, in such a way that the set of permutations obtained for all acyclic orientations of~$\cH$ is a zigzag language.
By Theorem~\ref{thm:jump} we can thus apply Algorithm~J to generate this zigzag language in Gray code order, and in a final step we interpret the jumps in permutations performed by Algorithm~J in terms of flip operations on acyclic orientations of~$\cH$.
The key insight that makes this work is that when vertices are in hyperfect elimination order, the posets defined by the acyclic orientations have the \emph{unique parent-child} property, namely that every vertex has at most one parent and one child in the poset (Lemma~\ref{lem:upc-heo}).

To prove our results on quotients of acyclic reorientations lattices discussed in Section~\ref{sec:res2}, we proceed as follows:
Given a peo-consistent digraph~$D$, we label its vertices with $1,\ldots,n$ according to this ordering, and we encode any of its acyclic orientations as a permutation on~$[n]$; see Figure~\ref{fig:cong2}~(b).
For a given lattice congruence of the reorientation lattice of~$D$, we select a set of representatives, one permutation from each equivalence class, such that those representative permutations form a zigzag language; see Figure~\ref{fig:cong2}~(c).
We show that for peo-consistent digraphs, the equivalence classes of any lattice congruence have a simple projection property that enables selecting representatives in an inductive `zigzag'-like way.
It then follows that the jumps in permutations performed by Algorithm~J correspond to steps along cover edges of the lattice quotient.

\section{Acyclic orientations of graphs}
\label{sec:aog}

In this section we review the structure of acyclic orientations of graphs and the associated combinatorial and geometric objects.
We also describe how the Savage-Squire-West Gray code for acyclic orientations of a chordal graph described in Section~\ref{sec:acyclic} can be cast as an instance of Algorithm~J, and we show how to implement this Gray code efficiently (Theorem~\ref{thm:SSW-algo} below).

\subsection{Poset preliminaries}
\label{sec:prelim}

We first recall some terminologies for a partially ordered set~$(P, <)$ that will be used throughout this paper.
A \emph{cover relation} is a pair~$x,y \in P$ with $x<y$ such that there is no~$z \in P$ with $x<z<y$.
In that case, we say that \emph{$y$ covers~$x$}, or \emph{$x$ is covered by~$y$}.
The \emph{cover graph} of $P$ has the elements of $P$ as vertices, and an edge between every pair of vertices that are in a cover relation.
A \emph{linear extension} of~$P$ is a total order that respects the comparabilities (or cover relations) of~$P$.
We write $\ext(P)$ for the set of linear extensions of~$P$.
An \emph{interval} $[x,y]$ of~$P$ is the set of all~$z$ in~$P$ such that $x<z<y$.

A poset~$(P,<)$ is a \emph{lattice}, if for every pair $x,y \in P$ there is a unique minimal element~$z$ such that~$z>x$ and~$z>y$, called the \emph{join~$x\vee y$ of~$x$ and~$y$}, and a unique maximal element~$z$ such that~$z<x$ and~$z<y$, called the \emph{meet~$x\wedge y$ of~$x$ and~$y$}.

\subsection{Flips in acyclic orientations of graphs}

We use standard terminology for digraphs~$D=(V,A)$, such as \emph{out-neighbor}, \emph{in-neighbor}, as well as \emph{out-degree} and \emph{in-degree} of a vertex~$v\in V$, denoted~$d^+(v)$ and~$d^-(v)$, respectively.
A \emph{source} is a vertex with zero in-degree, and a \emph{sink} is a vertex with zero out-degree.
The \emph{transitive reduction} $T_D$ of a digraph~$D$ is the digraph obtained by removing from~$D$ all arcs that can be obtained by applying the transitivity rule in~$D$; see Figure~\ref{fig:acyclic}.
When interpreting~$D$ as a poset, $T_D$ is the cover graph of~$D$.

An \emph{orientation}~$D$ of a graph~$G$ is a digraph obtained by orienting every edge of~$G$ in one of two ways.
An orientation of~$G$ is called \emph{acyclic} if it does not contain any directed cycles.
We write $\AO_G$ for the set of all acyclic orientations of the graph~$G$.

We refer to the operation of reversing the direction of a single arc of an orientation as an \emph{arc flip}.
We define a \emph{flip graph} on~$\AO_G$ by joining two acyclic orientations of~$G$ with an edge if and only if they differ in an arc flip.
Note that a transitive arc cannot be flipped, as this would create a directed cycle.
Conversely, if flipping an arc creates a directed cycle, then before the flip the arc was transitive.
Therefore, for every $D\in\AO_G$, an arc is flippable in~$D$ if and only if it belongs to the transitive reduction~$T_D$, and the degree of~$D$ in the flip graph is equal to the number of arcs in~$T_D$.

We often consider a total ordering of the vertex set~$V$ of a graph or digraph, and then we simply use $V=[n]=\{1,2,\ldots,n\}$.

\subsection{Graphical zonotopes}
\label{sec:zono}

The \emph{graphical arrangement} of $G=([n],E)$ is the collection of hyperplanes $\{H_{ij} \mid ij\in E\}$ in~$\mathbb{R}^n$, defined as~$H_{ij} := \{x\in\mathbb{R}^n \mid x_i = x_j\}$.
This arrangement defines the \emph{graphical fan} of $G$, the full-dimensional cones of which are in one-to-one correspondence with acyclic orientations of $G$.
The \emph{graphical zonotope~$Z(G)$} of~$G$ is the dual polytope of the graphical fan of $G$.
Original developments about these structures can be found in Greene~\cite{G77} and Greene and Zaslavsky~\cite{MR712251}; see also Stanley~\cite{MR2383131}.

The graphical zonotope~$Z(G)$ of~$G$ can be defined as the Minkowski sum of the line segments~$[e_i, e_j]$ for all $ij\in E$, where $e_i$ is the $i$th canonical basis vector of~$\mathbb{R}^n$. We can also define~$Z(G)$ as follows.
For an orientation (not necessarily acyclic) $D$ of~$G$, we consider the \emph{in-degree sequence} of~$D$, defined as $\delta_D:=(d^-(1),\ldots,d^-(n))\in\mathbb{N}^n$.
Then the graphical zonotope of~$G$ is
\begin{equation*}
Z(G) = \conv \{ \delta_D \mid D \text{\ orientation\ of\ } G \}.
\end{equation*}
It can be shown that $\delta_D$ is a vertex of~$Z(G)$ if and only if $D$ is acyclic,
hence the above definition remains valid if we restrict~$D$ to be an acyclic orientation of~$G$.
For a simple proof of this fact, we refer to that of a more general statement given as Proposition~3.9 in~\cite{rehberg_2021}.
The following lemma is a simple consequence of the definitions (see also the discussion in~\cite{MR1609342}).

\begin{lemma}
The flip graph on acyclic orientations~$\AO_G$ of a graph~$G$ is isomorphic to the skeleton of its zonotope~$Z(G)$.
\end{lemma}

A \emph{partial cube} is a graph that has an isometric embedding in the graph of a hypercube.
Dual graphs of hyperplane arrangements are known to be partial cubes (see for instance Chapter~7 in~\cite{O11}).
This applies in particular to graphical arrangements, and directly yields the following statement.
\begin{lemma}
The flip graph on acyclic orientations~$\AO_G$ of a graph~$G$ is a partial cube.
\end{lemma}
As a consequence, the \emph{flip distance} between two acyclic orientations of $G$, i.e.,
the minimum number of arc flips needed to transform one into the other, is equal to the number of edges oriented oppositely in the two orientations.

\subsection{Acyclic orientations of chordal graphs}

For a graph~$G=([n],E)$ and an integer $i\leq n$, we write $G_i$ for the subgraph of $G$ induced by the vertices in~$[i]$.
Similarly, for a digraph~$D$ with vertex set~$[n]$ and an integer~$i\leq n$, we write $D_i$ for the subdigraph of~$D$ induced by~$[i]$.
A graph is \emph{chordal} if every induced cycle has length~3.
A vertex whose neighborhood forms a clique is called \emph{simplicial}.
A graph $G=([n],E)$ is in \emph{perfect elimination order} if for all~$i\in [n]$, the vertex~$i$ is simplicial in~$G_i$.
It is well known that a graph is chordal if and only if it is isomorphic to a graph~$([n],E)$ in perfect elimination order~\cite{MR186421}.

We say that a graph $G=([n],E)$ has the \emph{unique parent-child property} if either $n=0$, or $n\geq 1$ and the following two conditions are satisfied:
\begin{enumerate}[label=(\roman*),leftmargin=8mm, noitemsep, topsep=1pt plus 1pt]
\item \label{itm:upc1} for every acyclic orientation~$D\in\AO_G$ the vertex~$n$ has in-degree and out-degree at most~1 in the transitive reduction $T_D$ of $D$;
\item the graph~$G_{n-1}$ has the unique parent-child property.
\end{enumerate}

\begin{lemma}
\label{lem:upc-peo}
A graph $G=([n],E)$ has the unique parent-child property if and only if it is in perfect elimination order.
\end{lemma}

\begin{proof}
$(\Leftarrow)$ By definition, the vertex~$n$ is simplicial in $G$, hence $n$ together with its neighbors induce a clique.
In any acyclic orientation $D\in\AO_G$, the transitive reduction of this clique is a path.
Therefore, the vertex~$n$ is involved in at most two arcs of~$T_D$ and has in-degree and out-degree at most~1.
The same holds for the vertex~$i$ in~$G_i$, for all $i\in [n]$.

\noindent$(\Rightarrow)$ Suppose that $G$ is not in perfect elimination order.
Without loss of generality, suppose that there are two nonadjacent neighbors~$a,b$ of~$n$ in~$G$.
Consider an orientation $D$ of $G$ in which
\begin{itemize}[leftmargin=5mm, noitemsep, topsep=1pt plus 1pt]
\item every arc having $n$ as endpoint is directed towards~$n$;
\item any other arc having~$a$ or~$b$ as endpoint is directed towards~$a$ or~$b$, respectively;
\item all other arcs $ij\in E$ are directed towards~$\max\{i,j\}$.
\end{itemize}
Then $D$ is an acyclic orientation of~$G$ such that the arcs $a\rightarrow n$ and $b\rightarrow n$ are both present in~$T_D$, contradicting condition~\eqref{itm:upc1} of the unique parent-child property.
\end{proof}

\subsection{Savage-Squire-West as an instance of Algorithm~J}

Recall the Savage-Squire-West Gray code for acyclic orientations of a chordal graph in perfect elimination order described in Section~\ref{sec:acyclic}.
We now show how to derive this Gray code as a special case of Algorithm~J in the Hartung-Hoang-M\"utze-Williams framework introduced in Section~\ref{sec:framework}.
For this purpose, we map acyclic orientations of a graph to permutations so that the image of all acyclic orientations forms a zigzag language of permutations.
This mapping is illustrated in Figures~\ref{fig:cong2} and~\ref{fig:ladder}.

\begin{lemma}
\label{lem:peo2perm}
Let $G=([n],E)$ be a graph in perfect elimination order.
With any acyclic orientation $D\in\AO_G$ we associate a permutation $\pi_D\in S_n$ as follows:
If $n=0$ then $\pi_D:=\varepsilon$, and if $n\geq 1$ we consider three cases:
\begin{enumerate}[label=(\roman*),leftmargin=8mm, noitemsep, topsep=1pt plus 1pt]
\item if the vertex~$n$ is a sink in $D$, then $\pi_D:=c_n(\pi_{D_{n-1}})$;
\item if the vertex~$n$ is a source in $D$, then $\pi_D:=c_1(\pi_{D_{n-1}})$;
\item otherwise, $\pi_D:=c_i(\pi_{D_{n-1}})$, where $i$ is the position in $\pi_{D_{n-1}}$ of the unique out-neighbor of $n$ in the transitive reduction $T_D$.
\end{enumerate}
Then the map $\AO_G\to S_n:D\mapsto \pi_D$ is injective, and
\begin{equation}
\label{eq:PiG}
\Pi_G:=\{\pi_D \mid D\in\AO_G\}
\end{equation}
is a zigzag language of permutations.
\end{lemma}

If the vertex~$n$ is isolated in~$G$, then it is both a sink and a source, in which case we use the encoding stated under~(i), and then the special condition~(z2) in the definition of zigzag languages applies.
Observe that $\pi_D$ is a linear extension of the poset defined by the transitive closure of~$D$ (whose cover graph is~$T_D$), and that the orientation~$D$ can be retrieved from $\pi_D$ by orienting every edge~$ij$ of $G$ towards the vertex in~$\{i,j\}$ that is further to the right in~$\pi_D$.
Also note that for any acyclic orientation~$D$ in which vertex~$i$ is a source or a sink in~$D_i$ for all~$i\in[n]$, the permutation~$\pi_D$ is peak-free.
In particular, if $i$ is a sink in~$D_i$ for all~$i\in[n]$, then $\pi_D=\ide_n$ is the identity permutation.
As remarked before, any of those permutations can serve as initial permutation~$\pi_0$ for Algorithm~J.

\begin{theorem}
For every graph $G = ([n],E)$ in perfect elimination order, Algorithm~J with input~$\Pi_G$ as defined in~\eqref{eq:PiG} generates a sequence of permutations $\pi_{D_1}, \pi_{D_2}, \ldots$, where $D_1,D_2,\ldots \in \AO_G$ such that $D_1,D_2,\ldots $ is a Hamilton path in the flip graph on acyclic orientations of~$G$, or equivalently, on the skeleton of the graphical zonotope~$Z(G)$.
\end{theorem}

\subsubsection{Efficient implementation}
\label{sec:SSW-algo}

We can now use the history-free implementation of Algorithm~J developed in~\cite[Sec.~5.1+8.7]{perm_series_iii} to compute the Savage-Squire-West Gray code efficiently; see the pseudocode stated as Algorithm~A.
In the following, for a chordal graph~$G=([n],E)$ in perfect elimination order, we write $N_i$ for the set of neighbors of~$i$ in~$G_i$.
For simplicity, the algorithm assumes that~$G$ is connected.
We remark that the case of disconnected chordal graphs~$G$ can be handled with slight adjustments, yielding the same runtime guarantees.
For details, see our C++ implementation~\cite{cos_orient}.

The algorithm maintains the current acyclic orientation~$D=([n],A)$ of~$G$ as an $n\times n$ adjacency matrix~$A$ with entries $a_{i,j}\in\{0,1\}$, where $a_{i,j}=1$ if and only if the arc $i\rightarrow j$ is present in~$D$.

The initial acyclic orientation~$D$ of~$G$ defined in step~A1 makes the vertex~$j$ a sink in~$D_j$ for all $j\in [n]$.
The corresponding permutation~$\pi_D$ is the identity permutation~$\pi_D=\ide_n$.
After the initialization in step~A1, the algorithm loops through steps~A2--A6, where A2 visits the current acyclic orientation~$D$, and steps~A3--A6 update the data structures before the next visit.
In step~A3, the auxiliary array~$s=(s_1,\ldots,s_n)$ is used to determine which vertex~$j$ is selected to have one of the arcs incident with~$j$ in~$D_j$ being flipped.
The array~$o=(o_1,\ldots,o_n)$ keeps track, for each vertex~$j\in[n]$, whether its current zigzag movement is from being sink to being source, in which case $o_j=\dirl$, or from source to sink, in which case $o_j=\dirr$; recall Figure~\ref{fig:acyclic-flip}.

Once a vertex~$j$ has become source or sink in~$D_j$, before flipping any arcs incident with it, in step~A4 the algorithm determines the transitive reduction of the clique in~$D_j$ induced by the vertices in~$N_j$.
Computing the transitive reduction (which is a path, i.e., a total order) amounts to sorting the set~$N_j$, using for comparisons the orientations of the arcs between vertices~$i,i'\in N_j$, which can be queried in constant time by reading the entries~$a_{i,i'}$ of the adjacency matrix.
This sorting happens exactly once at the beginning of each zigzag movement, and the resulting total order is stored in the array~$T_j$.
In each execution of step~A5, the next arc incident with~$j$ in the precomputed list~$T_j$ is flipped, using the index~$t_j$ into the list~$T_j$.
Once $t_j$ reaches the maximum value~$t_j=|N_j|$, i.e., the last entry of the list~$T_j$, which means that~$j$ has now become a source or sink in~$D_j$, then the direction~$o_j$ of the zigzag movement for~$j$ is reversed in step~A6, and the array~$s$ is updated accordingly (see~\cite{perm_series_iii} for details).

\begin{algo}{\bfseries Algorithm~A}{History-free arc flips}
Given a connected chordal graph~$G=([n],E)$ in perfect elimination order, this algorithm generates all acyclic orientations of~$G$ by arc flips in the order given by the Savage-Squire-West construction (recall Section~\ref{sec:acyclic}).
It maintains the current acyclic orientation~$D=([n],A)$ of~$G$ as an adjacency matrix~$A=(a_{i,j})_{i,j\in[n]}$, with $a_{i,j}\in\{0,1\}$, total orderings~$T_j$ of $N_j$ and an index~$t_j$, $0\leq t_j\leq |N_j|$, into the array~$T_j$ for all $j=1,\ldots,n$, as well as auxiliary arrays $o=(o_1,\ldots,o_n)$ and $s=(s_1,\ldots,s_n)$.
\begin{enumerate}[label={\bfseries A\arabic*.}, leftmargin=8mm, noitemsep, topsep=3pt plus 3pt]
\item{} [Initialize] For $i,j=1,\ldots,n$ set $a_{i,j}\gets 0$.
Then for $j=1,\ldots,n$, orient all arcs incident with~$j$ in~$D_j$ towards~$j$, i.e., for all $i\in N_j$ set $a_{i,j}\gets 1$.
Also set $t_j\gets 0$, $o_j\gets \dirl$, and $s_j\gets j$ for $j=1,\ldots,n$.
\item{} [Visit] Visit the current acyclic orientation~$D=([n],A)$.
\item{} [Select vertex] Set $j\gets s_n$, and terminate if $j=1$.
\item{} [Sort neighbors] If $t_j=0$, compute the transitive reduction $v_1\rightarrow v_2\rightarrow \cdots\rightarrow v_k$ of the clique in~$D_j$ formed by the vertices in~$N_j$ via sorting, using the adjacency matrix entries~$A_{i,i'}$, $i,i'\in N_j$, for comparisons.
If $o_j\gets \dirl$, set $T_j\gets (v_k,v_{k-1},\ldots,v_1)$, and if $o_j=\dirr$ set $T_j\gets (v_1,v_2,\ldots,v_k)$.
\item{} [Flip arc] Set $t_j\gets t_j+1$.
In the current acyclic orientation~$D$, flip the arc between~$j$ and~$i\gets T_{j,t_j}$, i.e., if $o_j=\dirl$ set $a_{i,j}\gets 0$ and $a_{j,i}\gets 1$, whereas if $o_j=\dirr$ set $a_{i,j}\gets 1$ and $a_{j,i}\gets 0$.
\item{} [Update $o$ and $s$] Set $s_n\gets n$.
If $t_j=|N_j|$, then if $o_j=\dirl$ ($j$ has become source in~$D_j$) set $o_j\gets \dirr$, and if $o_j=\dirr$ ($j$ has become sink in~$D_j$) set $o_j\gets \dirl$, and in both cases set $t_j\gets 0$, $s_j\gets s_{j-1}$ and $s_{j-1}\gets j-1$. Go back to~A2.
\end{enumerate}
\end{algo}

\begin{theorem}
\label{thm:SSW-algo}
Given a connected chordal graph~$G=([n],E)$ in perfect elimination order, Algorithm~A visits each acyclic orientation of~$G$ in time~$\cO(\log \omega)$ on average, where $\omega=\omega(G)$ is the clique number of~$G$.
\end{theorem}

\begin{proof}
Steps~A2, A4, A5 and~A6 clearly take only constant time.
The sorting step~A4 takes time~$\cO(d\log d)$, where $d=|N_j|$ is the degree of the vertex~$j$ in~$D_j$.
This iteration of the main loop is followed by~$d-1$ iterations later iterations of the main loop in which vertex~$j$ is considered but step~A4 is skipped because~$t_j>0$ (specifically, this happens for $t_j\in\{1,\ldots,d-1\}$.
So overall the algorithm visits~$d$ acyclic orientations in time~$\cO(d\log d)$, which is $\cO(\log d)$ on average.
Clearly, we have $d\leq \omega(G)$.
\end{proof}

The space required by Algorithm~A to store the adjacency matrix of~$G$ is clearly~$\cO(n^2)$, and the initialization time spent in step~A1 is~$\cO(n^2)$.
Testing whether an arbitrary graph~$G$ is chordal, and if so computing a perfect elimination ordering for~$G$, can be done in time~$\cO(m+n)$ by lexicographic breadth-first-search~\cite{MR408312}, where $m$ is the number of edges of~$G$.
Clearly, $\cO(m+n)$ is dominated by the initialization time~$\cO(n^2)$ of Algorithm~A.

\section{Acyclic orientations of hypergraphs}
\label{sec:aoh}

In this section we establish our first main result, a Gray code for acyclic orientations of certain hypergraphs (Theorem~\ref{thm:mainhyper} below).
The hypergraphs we consider admit a vertex ordering that we refer to as hyperfect elimination order.
When specialized to graphs, this corresponds to a perfect elimination order, and we recover the Savage-Squire-West construction.
When specialized to graphical building sets, we recover the Gray code for elimination trees of chordal graphs described in~\cite{DBLP:conf/soda/CardinalMM22} (Lemma~\ref{lem:BG-elim}).
The algorithm also applies to chordal building sets (Theorem~\ref{thm:mainbuild}) and yields Hamilton paths on the corresponding hypergraphic polytopes called chordal nestohedra.

\subsection{Flips in acyclic orientations of hypergraphs}

Let $\cH=(V,\cE)$ be a hypergraph, where $\cE\seq 2^V$.
An \emph{orientation} of $\cH$ is a function $h:\cE\to V$ such that $h(A)\in A$ for every $A\in\cE$; see Figure~\ref{fig:hyper}~(a).
We refer to $h(A)$ as the \emph{head} of hyperedge $A$ in the orientation.\footnote{Orientations of hypergraphs have been defined differently in similar contexts. In particular, the definition of hypergraph orientation used by Benedetti, Bergeron, and Machacek~\cite{MR3960512} is more general than ours. In their terminology, we restrict to orientations with heads of size one only. The general definition is most useful for a complete characterization of the faces of the hypergraphic polytope. Rehberg~\cite{rehberg_2021} refers to our definition as a \emph{heading} instead of an orientation.}
A \emph{path} in an orientation~$h$ of~$\cH$ is a sequence~$(v_1,\ldots,v_k)$ of distinct vertices for which there are hyperedges~$A_1,\ldots,A_{k-1}\in\cE$ such that $v_i,v_{i+1}\in A_i$ and $h(A_i)=v_{i+1}$ for all $i=1,\ldots,k-1$.
A~\emph{cycle} is such a path with $k\geq 2$ and the additional property that $v_k,v_1\in A_k$ and $h(A_k)=v_1$ for some hyperedge~$A_k\in\cE$.
By this definition a loop does not count as a cycle.
An orientation of~$\cH$ is \emph{acyclic} if it does not contain any cycles.
Equivalently, $\cH$ is acyclic if the digraph formed by all arcs~$i\rightarrow j$ for every pair of distinct vertices~$i,j\in V$ with $i,j\in A$ and $j=h(A)$ for some hyperedge~$A\in \cE$ is acyclic; see Figure~\ref{fig:hyper}~(b).
We write $\AO_\cH$ for the set of all acyclic orientations of the hypergraph~$\cH$.

Given an orientation~$h$ of a hypergraph~$\cH=([n],\cE)$, a \emph{pair flip} involves a pair of distinct vertices~$(i,j)$, $i,j\in [n]$, and maps an orientation~$h$ to a distinct orientation~$h'$ such that
\begin{equation}
\label{eq:pair-flip}
h'(A) :=
\begin{cases}
  i & \text{\ if\ } h(A) = j \text{\ and\ } i\in A, \\
  h(A) & \text{\ otherwise,}
\end{cases}
\end{equation}
for all $A\in\cE$; see Figure~\ref{fig:hyper2}~(c).
Note that in order for~$h'$ to be distinct from~$h$, the definition requires that there must exist a hyperedge~$A\in\cE$ with $h(A)=j$ (which then satisfies $h'(A)=i$).
We define a \emph{flip graph} on $\AO_\cH$ by joining two acyclic orientations of~$\cH$ with an edge if and only if they differ in a pair flip.

The following lemma identifies flippable pairs in an acyclic orientation of a hypergraph.
It generalizes the situation for acyclic orientations of graphs, where flippable arcs were precisely the arcs in the transitive reduction.
An acyclic orientation~$h$ of a hypergraph~$\cH=([n],\cE)$ yields a poset on~$[n]$ defined as the transitive closure of the relation
\begin{equation*}
i\prec j \Longleftrightarrow \text{there exist a hyperedge $A\in \cE$ with $i,j\in A$ and $j=h(A)$}.
\end{equation*}
We denote this poset by $P_{\cH,h}$, and we write $i\prec j$ to express comparabilities in this poset in infix notation; see Figure~\ref{fig:hyper}~(c).
Note that if the pair~$(\cH,h)$ is a simple digraph~$D$, then the poset~$P_{\cH,h}$ is the poset defined by the transitive closure of~$D$.

\begin{lemma}
\label{lem:flippable}
A pair~$(i,j)$ is flippable in an acyclic orientation~$h$ of~$\cH$ if and only if $j$ covers~$i$ in the poset~$P_{\cH,h}$.
\end{lemma}

\begin{proof}
$(\Leftarrow)$ Suppose that $j$ covers~$i$ in~$P_{\cH,h}$ and let $h'$ be the orientation obtained after flipping the pair~$(i,j)$.
Suppose for the sake of contradiction that $h'$ is not acyclic.
Then there must exist a path~$(v_1,\ldots,v_k)$, $k\geq 3$, with $v_1=i$ and $v_k=j$ in the orientation~$h$ of~$\cH$.
But this implies that the relation~$i\prec j$ is obtained by transitivity, i.e., $j$ does not cover~$i$ in~$P_{\cH,h}$.

$(\Rightarrow)$ Suppose that $(i,j)$ is flippable, and suppose for the sake of contradiction that $j$ does not cover $i$ in~$P_{\cH,h}$.
Then the relation $i\prec j$ is obtained by transitivity, i.e., there is a path $(v_1,\ldots,v_k)$, $k\geq 3$, with $v_1=i$ and $v_k=j$ in the orientation~$h$ of~$\cH$.
After flipping~$(i,j)$, this path creates a cycle in the resulting orientation~$h'$.
Specifically, if there is a hyperedge~$A\in\cE$ with $i,j\in A$ and $v_{k-1}\in A$, then $(v_1,\ldots,v_{k-1})$ is a cycle in~$h'$, and otherwise $(v_1,\ldots,v_k)$ is a cycle in~$h'$.
\end{proof}

Note that we can define a surjective map $S_n\to\AO_\cH$.
For a permutation $\pi\in S_n$, every hyperedge $A\in\cE$ can be oriented towards $h(A):=\argmax_{i\in A} \pi^{-1}(i)$, i.e., towards the element from~$A$ that appears rightmost in~$\pi$.
Every acyclic orientation $h\in\AO_\cH$ can be obtained in this way.
Any linear extension $\pi\in\ext (P_{\cH,h})$ is such that orienting the hyperedges according to $\pi$ yields the orientation~$h$.
The set $\{\ext (P_{\cH,h}) \mid h\in\AO_\cH\}$ is therefore a partition of~$S_n$ into equivalence classes.

\subsection{Hypergraphic polytopes}

Generalizing the situation for acyclic orientations of graphs described in Section~\ref{sec:zono}, the acyclic orientations of a hypergraph~$\cH$ are in one-to-one correspondence with the vertices of a polytope associated with~$\cH$, and the flip graph on acyclic orientations is isomorphic to the skeleton of this polytope.

The \emph{hypergraphic polytope}~\cite{MR3960512,aguiar_ardila_2017} of a hypergraph $\cH=([n],\cE)$ can be defined as a Minkowski sum of simplices.
Specifically, with a subset $S\seq [n]$, we associate the standard simplex $\Delta_S := \conv \{e_i \mid i\in S\}$, and the hypergraphic polytope $Z(\cH)$ of~$\cH$ can be defined as $Z(\cH)=\sum_{A\in\cE} \Delta_A$.
Hypergraphic polytopes can also be defined as convex hulls of in-degree sequences.
For an orientation~$h$ of the hypergraph~$\cH$, let $d^-(i):=|\{A\in\cE \mid h(A)=i\}|$ be the \emph{in-degree} of vertex~$i$.
The vector $\delta_{\cH,h}:=(d^-(1),\ldots,d^-(n))\in\mathbb{N}^n$ is the \emph{in-degree sequence} of~$h$.
Then the hypergraphic polytope of~$\cH$ is
\begin{equation}
\label{eq:hypervert}
Z(\cH)=\conv \{\delta_{\cH,h} \mid \text{$h$ orientation of $\cH$} \}.
\end{equation}
Again, this definition does not change if we require the orientations~$h$ to be acyclic.
A proof of these facts is implicit in previous works~\cite{MR3960512} and spelled out by Rehberg~\cite{rehberg_2021} (Proposition 3.9).

The following lemma is a special case of Theorem~2.18 in~\cite{MR3960512}.

\begin{lemma}[\cite{MR3960512}]
\label{lem:hyperflip}
The flip graph on acyclic orientations~$\AO_\cH$ of a hypergraph~$\cH$ is isomorphic to the skeleton of its hypergraphic polytope~$Z(\cH)$.
\end{lemma}

\subsection{Hypergraphs in hyperfect elimination order}

We now generalize the notion of perfect elimination order to hypergraphs~$\cH$, and prove that this order is a necessary and sufficient condition for the unique parent-child property to hold in the poset~$P_{\cH,h}$ of any acyclic orientation $h\in\AO_\cH$.

\subsubsection{Hyperfect elimination order}
\label{sec:hyperfect}

For a hypergraph~$\cH=([n],\cE)$ and an integer~$i\leq n$, we write $\cH_i$ for the subgraph of $\cH$ induced by the vertices in~$[i]$.

$\cH=([n],\cE)$ is in \emph{hyperfect elimination order} if either $n=0$, or $n\geq 1$ and the following two conditions are satisfied:
\begin{enumerate}[label=(\roman*),leftmargin=8mm, noitemsep, topsep=1pt plus 1pt]
\item \label{itm:heo1}
For any two hyperedges $A,B\in\cE$ with $n\in A\cap B$ and any two distinct vertices $a\in A-n$, $b\in B-n$ there is a hyperedge~$X\in\cE$ such that $\{a,b\} \seq X \seq (A \cup B)-n$;
\item \label{itm:heo2}
$\cH_{n-1}$ is in hyperfect elimination order.
\end{enumerate}

We emphasize that the hyperedges~$A,B\in\cE$ in this definition are not required to be distinct, i.e., this condition must hold also for hyperedges~$A=B$.

Observe that if the hypergraph is a graph, thus if all hyperedges have size two, then the first condition states that the neighbors of~$n$ must be pairwise adjacent, hence that $n$ is simplicial.
The hyperfect elimination order is therefore a generalization of the perfect elimination order for chordal graphs.
(For other generalizations of perfect elimination orders to hypergraphs, see e.g.~\cite{MR2603461}.)

\subsubsection{Unique parent-child property for hypergraphs}

We now generalize Lemma~\ref{lem:upc-peo} to hypergraphs in hyperfect elimination order.

We say that a hypergraph $\cH=([n],\cE)$ has the \emph{unique parent-child property} if either $n=0$, or $n\geq 1$ and the following two conditions are satisfied:
\begin{enumerate}[label=(\roman*),leftmargin=8mm, noitemsep, topsep=1pt plus 1pt]
\item for every acyclic orientation~$h\in\AO_\cH$ the vertex~$n$ covers at most one element and is covered by at most one element in~$P_{\cH,h}$;
\item the hypergraph $\cH_{n-1}$ has the unique parent-child property.
\end{enumerate}

\begin{lemma}
\label{lem:upc-heo}
A hypergraph $\cH=([n],\cE)$ has the unique parent-child property if and only if it is in hyperfect elimination order.
\end{lemma}

\begin{proof}
$(\Leftarrow)$
Suppose that $\cH$ is in hyperfect elimination order.
Let $h$ be an acyclic orientation of $\cH$ and $a,b\in [n-1]$ two distinct vertices.
We first show that $n$ cannot cover~$a$ and~$b$ in~$P_{\cH,h}$.
Suppose for the sake of contradiction that $n$ covers both~$a$ and~$b$ in~$P_{\cH,h}$.
This means there are hyperedges $A,B\in \cE$ such that $\{n,a\}\seq A$, $\{n,b\}\seq B$ and $h(A)=h(B)=n$.
By the definition of hyperfect elimination order, it follows that there is a hyperedge~$X \in \cE$ such that $\{a,b\}\seq X \seq (A\cup B)-n$.
Note that $h(X)\notin \{a,b\}$; otherwise we have either $a\prec b$ or~$b\prec a$ which implies that $n$ cannot cover both~$a$ and~$b$, a contradiction.
As $h(X) \in (A\cup B)-n$, we also conclude that $h(X)\prec n$.
Thus, $a\prec h(X)\prec n$ and $b\prec h(X)\prec n$ which means that $n$ does not cover~$a$ nor~$b$, a contradiction.

We now show that $n$ cannot be covered by~$a$ and~$b$.
Suppose for the sake of contradiction that $n$ is covered by both~$a$ and~$b$ in~$P_{\cH,h}$.
This means there are hyperedges $A,B \in \cE$ such that $\{n,a\}\seq A$, $\{n,b\}\seq B$, $h(A)=a$ and $h(B)=b$.
By the definition of hyperfect elimination order, it follows that there is a hyperedge~$X \in \cE$ such that $\{a,b\}\seq X \seq (A\cup B)-n$.
Note that if $h(X)\in A$, then we must have~$h(X)=a$.
Indeed, if $h(X)\neq a$, then we would have $a\prec h(X)$ and $h(X)\prec a$ as $h(A)=a$ and $h(X)\in A$, a contradiction.
Symmetrically, if $h(X)\in B$, then we must have~$h(X)=b$.
It follows that $h(X)\in \{a,b\}$.
Hence, we either have $a\prec b$ or $b\prec a$, which implies that $n$ cannot be covered by both~$a$ and~$b$, a contradiction.

Combining these observations, we obtain that $n$ covers at most one element and is covered by at most one element in~$P_{\cH,h}$.
By iterating this argument for~$\cH_{n-1}$, we conclude that $\cH$ has the unique parent-child property.

$(\Rightarrow)$
Suppose that $\cH$ has the unique parent-child property, and suppose for the sake of contradiction that $\cH$ is not in hyperfect elimination order.
Without loss of generality, suppose that there are hyperedges $A,B \in \cE$ with $n\in A\cap B$ and two distinct vertices $a \in A-n$, $b \in B-n$ such that
\begin{equation}
\label{eq:heo-neg}
\text{no hyperedge $X\in\cE$ satisfies $\{a,b\}\seq X\seq (A\cup B)-n$.}
\end{equation}
We construct an acyclic orientation~$h$ of~$\cH$ such that $n$ covers at least two elements in~$P_{\cH,h}$.
To do so, recall that a permutation~$\pi$ of $[n]$ induces an acyclic orientation~$h_\pi$ by defining $h_\pi(X):=\argmax_{i \in X} \pi^{-1}(i)$, and that~$\pi$ is a linear extension of~$P_{\cH,h_{\pi}}$.

Let $R:=(A \cup B) \setminus \{a,b,n\}$ and $S:=[n]\setminus (A\cup B)$, and note that $R,S,\{a\},\{b\},\{n\}$ is a partition of~$[n]$.
Consider the permutation~$\pi$ on~$[n]$ defined by
\begin{equation}
\label{eq:piRS}
\pi^{-1}(r)<\pi^{-1}(a)<\pi^{-1}(b)<\pi^{-1}(n)<\pi^{-1}(s)
\end{equation}
for all $r\in R$ and~$s\in S$, and let $h:=h_\pi$ be the acyclic orientation induced by~$\pi$.
We claim that $n$ covers both~$a$ and~$b$ in~$P_{\cH,h}$.
First note that $a,b\prec n$, as $h(A)=h(B)=n$.
Furthermore, there can be no $x\in[n]\setminus\{a,b\}$ with $a\prec x\prec n$ or $b\prec x\prec n$, as $\pi$ is a linear extension of~$P_{\cH,h}$.
We conclude that $n$ covers~$b$.
Furthermore, $n$ covers~$a$ unless $b$ covers~$a$.
However, if $b$ covers~$a$, then by~\eqref{eq:piRS} there must be a hyperedge~$X\in\cE$ with $\{a,b\}\seq X\seq (A\cup B)-n$ and $h(X)=b$, contradicting~\eqref{eq:heo-neg}.
This completes the proof.
\end{proof}

Note that in the case of graphs in perfect elimination order, the proof of Lemma~\ref{lem:upc-peo} used the fact that the neighbors of $n$ are always totally ordered in~$P_D$ (whose transitive reduction is~$T_D$).
This property is not true for hypergraphs in hyperfect elimination orders.
Consider for instance the hypergraph $\cH=([4],\{12, 123, 1234\})$, which is in hyperfect elimination order.
If we consider the acyclic orientation $h$ such that $h(12)=h(123)=1$ and $h(1234)=4$, we have that~2 and~3 are incomparable in~$P_{\cH,h}$.

\subsection{Generation algorithm}

\subsubsection{Generation of acyclic orientations of hypergraphs}

Using Lemma~\ref{lem:upc-heo}, we can describe a simple recursive algorithm generating a Hamilton path in the flip graph on acyclic orientations of a hypergraph in hyperfect elimination order; see Figure~\ref{fig:tree}.
This algorithm generalizes the Savage-Squire-West construction for chordal graphs~\cite{MR1267311} described in Section~\ref{sec:acyclic}.

\begin{figure}
\centering
\makebox[0cm]{ 
\includegraphics{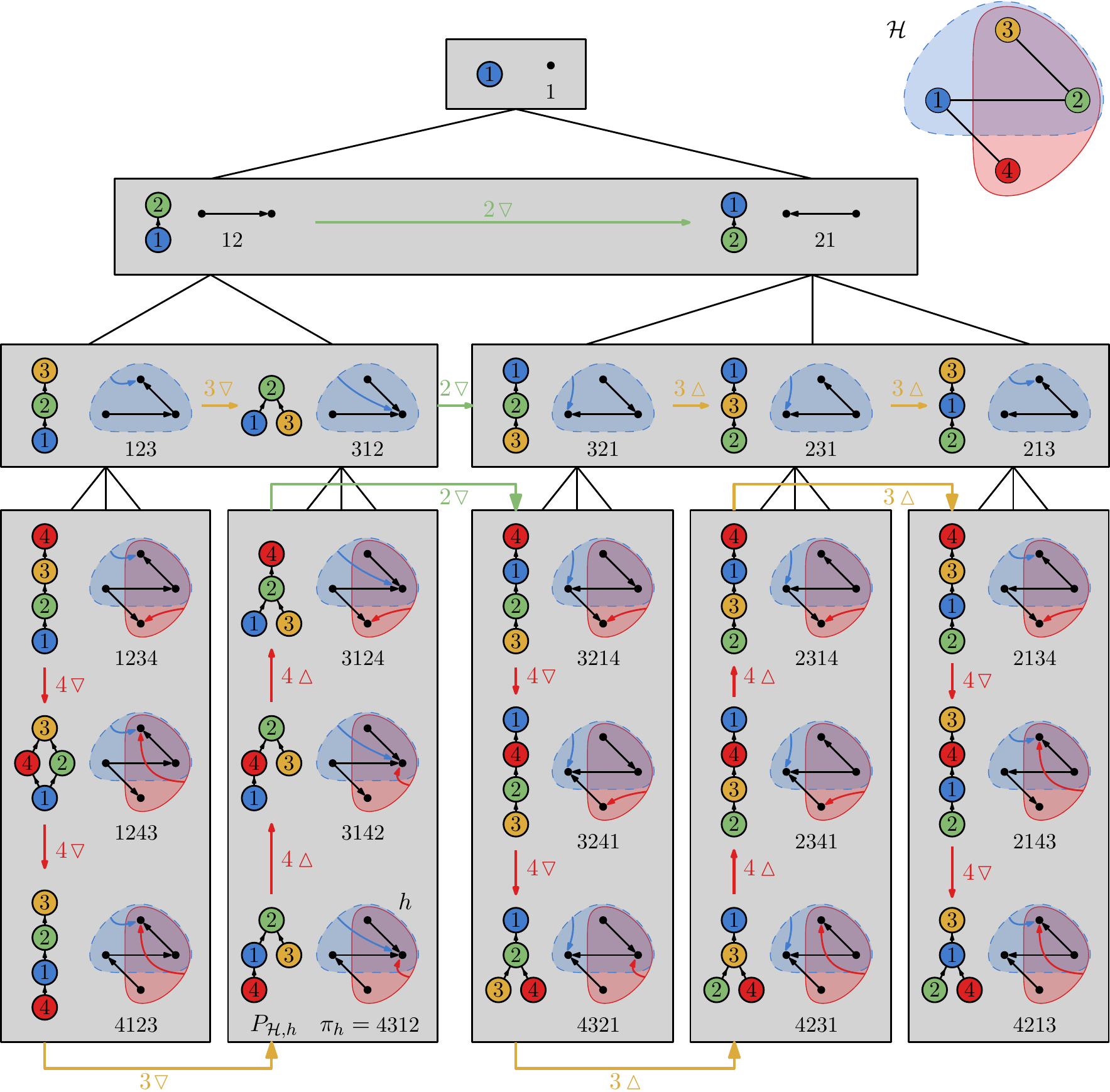}
}
\caption{Gray code of acyclic orientations of a hypergraph~$\cH$ in hyperfect elimination order.
The levels of the tree correspond to the induction steps.
At each step the figure shows the acyclic orientation~$h$ of~$\cH$, and the corresponding poset~$P_{\cH,h}$ and permutation~$\pi_h$.
The notation $j\dird$ indicates a pair flip~$(i,j)$ where $i$ is the unique child of~$j$ in~$P_{\cH_j,h_j}$ by Lemma~\ref{lem:upc-heo}.
Similarly, $j\diru$ indicates a pair flip~$(j,i)$ where $i$ is the unique parent of~$j$ in~$P_{\cH_j,h_j}$.
}
\label{fig:tree}
\end{figure}

We consider a hypergraph $\cH=([n],\cE)$ in hyperfect elimination order, and proceed by induction on~$n$.
The base case $n=0$ is trivial.
For $n\geq 1$, suppose we have a Hamilton path in the flip graph on acyclic orientations of~$\cH_{n-1}$.
Every acyclic orientation $h_{n-1}\in\AO_{\cH_{n-1}}$ can be extended to an acyclic orientation~$h\in\AO_\cH$ in two ways:
\begin{itemize}[leftmargin=5mm, noitemsep, topsep=1pt plus 1pt]
\item for all $A\in\cE$ such that $n\in A$, we define $h(A):=n$, in which case $n$ is maximal in~$P_{\cH,h}$;
\item consider a permutation $\pi\in\ext (P_{\cH_{n-1},h_{n-1}})$ and for all $A\in\cE$ such that $n\in A$, let $h(A):=\argmax_{i\in A-n} \pi^{-1}(i)$, in which case $n$ is minimal in~$P_{\cH,h}$.
\end{itemize}
Furthermore, we can get from the former to the latter, by iteratively flipping $n$ with the unique vertex it covers in this order, until there is no such vertex and $n$ becomes minimal in~$P_{\cH,h}$.
Conversely, we can move $n$ up by iteratively flipping $n$ and the unique vertex that covers it until $n$ becomes maximal in~$P_{\cH,h}$.
We can therefore replace any acyclic orientation in the Hamilton path in the flip graph on $\AO_{\cH_{n-1}}$ by a sequence of acyclic orientations in which vertex $n$ either moves down from maximal to minimal in~$P_{\cH,h}$, or up from minimal to maximal.
These sequences can be concatenated to yield a Hamilton path in the flip graph on $\AO_\cH$, in which $n$ `zigzags' alternately up and down in~$P_{\cH,h}$.

\subsubsection{Algorithm~J for acyclic orientations of hypergraphs}

The recursive algorithm described above can be cast as a special case of Algorithm~J in the Hartung-Hoang-M\"utze-Williams framework~\cite{MR4391718}.

\begin{lemma}
\label{lem:heo2perm}
Let $\cH=([n],\cE)$ be a hypergraph in hyperfect elimination order.
For an orientation $h\in\AO_\cH$, let $h_{n-1}$ be its restriction to~$\cH_{n-1}$.
With any acyclic orientation $h\in\AO_\cH$ we associate a permutation $\pi_h\in S_n$ as follows:
If $n=0$ then $\pi_h:=\varepsilon$, and if $n\geq 1$ we consider three cases:
\begin{enumerate}[label=(\roman*),leftmargin=8mm, noitemsep, topsep=1pt plus 1pt]
\item if the vertex~$n$ is maximal in~$P_{\cH,h}$, then $\pi_h:=c_n(\pi_{h_{n-1}})$;
\item if the vertex~$n$ is minimal in~$P_{\cH,h}$, then $\pi_h:=c_1(\pi_{h_{n-1}})$;
\item otherwise, $\pi_h := c_i(\pi_{h_{n-1}})$, where $i$ is the position in $\pi_{h_{n-1}}$ of the unique vertex that covers~$n$ in~$P_{\cH,h}$.
\end{enumerate}
Then the map $\AO_\cH\to S_n: h\mapsto \pi_h$ is injective, and
\begin{equation}
\label{eq:PiH}
\Pi_\cH := \{ \pi_h \mid h\in\AO_\cH \}
\end{equation}
is a zigzag language of permutations.
\end{lemma}

If the vertex~$n$ is isolated in~$P_{\cH,h}$, then it is both maximal and minimal, in which case we use the encoding stated under~(i), and then the special condition~(z2) in the definition of zigzag languages applies.
The definition of the mapping $h\mapsto \pi_h$ is illustrated in Figure~\ref{fig:tree}.
Lemma~\ref{lem:heo2perm} can be proved straightforwardly by induction; we omit the details.

Note that by definition, we have $\pi_h\in \ext (P_{\cH,h})$.
Therefore, the language $\Pi_\cH$ defines a set of representatives for the equivalence classes of permutations in~$S_n$.

\begin{lemma}
\label{lem:algJ-flip}
When running Algorithm~J with input~$\Pi_\cH$ as defined in~\eqref{eq:PiH}, then for any two permutations $\pi_h, \pi_{h'}$ that are visited consecutively, the corresponding acyclic orientations~$h$ and~$h'$ of~$\cH$ are adjacent in the flip graph on~$\AO_\cH$.
\end{lemma}

To prove Lemma~\ref{lem:algJ-flip}, we introduce the following lemma from~\cite{perm_series_iii}.
We say that a jump of a value~$j$ in a permutation~$\pi\in S_n$ is \emph{clean}, if for every $k=j+1,\ldots,n$, the value~$k$ is either to the left or right of all values smaller than~$k$ in~$\pi$.

\begin{lemma}[{\cite[Lemma~24~(d)]{perm_series_iii}}]
\label{lem:clean-jumps}
For any zigzag language $L_n\seq S_n$, all jumps performed by Algorithm~J are clean.
\end{lemma}

Furthermore, for proving Lemma~\ref{lem:algJ-flip} we establish the following auxiliary lemma.

\begin{lemma}
\label{lem:jump-flip}
Let $\cH=([n],\cE)$ be a hypergraph in hyperfect elimination order, and let $\cH_\nu=([\nu],\cE_\nu)$ for some $\nu \in [n]$.
Let $\pi'$ be obtained from~$\pi_{h_\nu}$ by a clean right jump of the value~$j=\pi_{h_\nu}(i)$ by $d$ steps.
Suppose that for all $a\in\{\pi_{h_\nu}(i+2),\ldots,\pi_{h_\nu}(i+d)\}$ there is no hyperedge~$A\in\cE_\nu$ with $j,a\in A$ and $h_\nu(A)=a$, and either $i+d=\nu$, or for $b:=\pi_{h_\nu}(i+d+1)$ we have $b>j$ or there is a hyperedge~$A\in\cE_\nu$ with $j,b\in A$ and $h_\nu(A)=b$.
Then the jump in~$\pi_{h_\nu}$ is minimal, we have $\pi' \in \Pi_{\cH_\nu}$ and the acyclic orientation~$h'\in\AO_{\cH_\nu}$ obtained from~$h_\nu$ by the pair flip~$(j,\pi_{h_\nu}(i+1))$ satisfies $\pi'=\pi_{h'}$.
\end{lemma}

\begin{proof}
We proceed by induction on~$\nu$.
The result is clear if $\nu=0$ or $\nu=1$, so we assume that~$\nu\geq 2$.
We distinguish two cases.

{\bf Case 1:} $j = \nu$.
Let $\hat \pi:=p(\pi_{h_\nu})=p(\pi')\in\Pi_{\cH_{\nu-1}}$.
By definition, we have that $\pi_{h_\nu}=c_i(\hat \pi)$ and $\pi_{h_\nu}(i+1)$ covers $\nu$ in~$P_{\cH_\nu,h_\nu}$.
Recall that by Lemma~\ref{lem:heo2perm} the mapping $h_\nu \mapsto \pi_{h_\nu}$ is a bijection between~$\AO_{\cH_\nu}$ and~$\Pi_{\cH_\nu}$ and that for every $A \in \cE_\nu$ we have
\begin{equation*}
h_\nu(A) = \argmax_{x \in A} \pi_{h_\nu}^{-1}(x).
\end{equation*}
Since $\pi_{h_\nu}(i+1)$ covers~$\nu$ in~$P_{\cH_\nu,h_\nu}$, by Lemma~\ref{lem:flippable} we can define $h':\cE_\nu \to [\nu]$ as the acyclic orientation that is obtained from~$h_\nu$ by the pair flip~$(\nu,\pi_{h_\nu}(i+1))$.
Using the definition~\eqref{eq:pair-flip}, we therefore obtain
\begin{equation}
\label{eq:h-flipped}
h'(A)=\begin{cases}
  \nu & \text{if $h_\nu(A)=\pi_{h_\nu}(i+1)$ and $\nu \in A$},\\
  h_\nu(A) & \text{otherwise}.
\end{cases}
\end{equation}

We aim to show that $\pi'$ is a linear extension of~$P_{\cH_\nu,h'}$.
Let $x,y \in [\nu]$ such that $x$ is covered by~$y$ in~$P_{\cH_\nu,h'}$.
As $x$ is covered by $y$, there exists a hyperedge $X\in \cE_\nu$ with $x,y \in X$ and $h'(X)=y$.

We first consider the case~$y\neq \nu$.
In this case we have $h'(X)=h_\nu(X)=y$.
Hence, $x\prec_{h_\nu} y$ and consequently $\pi_{h_\nu}^{-1}(x)<\pi_{h_\nu}^{-1}(y)$.
If $x\neq \nu$, then we trivially have $\pi_{h'}^{-1}(x)<\pi_{h'}^{-1}(y)$.
It remains to consider the case $x=\nu$.
We define
\begin{equation*}
\begin{split}
C'&:=\{\pi_{h_\nu}(i+2),\ldots,\pi_{h_\nu}(i+d)\}, \\
C&:=\{\pi_{h_\nu}(i+1)\}\cup C',
\end{split}
\end{equation*}
Note that $y\neq \pi_{h_\nu}(i+1)$, otherwise the definition~\eqref{eq:h-flipped} would imply $h'(X)=x$ and consequently $y\prec_{h'} x$.
Furthermore, from the assumption that for all $a\in C'$ there is no hyperedge~$A\in\cE_\nu$ with $\nu,a\in A$ and $h_\nu(A)=a$, we obtain that $y\notin C'$.
Combining these two observations shows that $y\notin C$, and therefore we have $\pi_{h'}^{-1}(x)<\pi_{h'}^{-1}(y)$, as desired.

We now consider the case~$y=\nu$.
There are two possible subcases.
If $\pi_{h_\nu}(i+1)\in X$, then by the definition~\eqref{eq:h-flipped} we have $h_\nu(X)=\pi_{h_\nu}(i+1)$ and consequently $\pi_{h_\nu}^{-1}(x)\leq i+1$, implying that $\pi_{h'}^{-1}(x)<\pi_{h'}^{-1}(y)$.
If $\pi_{h_\nu}(i+1)\notin X$, then we have $h'(X)=h_\nu(X)=y$, and consequently $\pi_{h_\nu}^{-1}(x)<i$, implying that $\pi_{h'}^{-1}(x)<\pi_{h'}^{-1}(y)$ as well.

We have shown that $\pi'$ is a linear extension of~$P_{\cH_\nu,h'}$.
We now use the assumption that either $i+d=\nu$, or for $b:=\pi_{h_\nu}(i+d+1)$ there is a hyperedge~$A\in\cE_\nu$ with $\nu,b\in A$ and $h_\nu(A)=b$ (the case $b>j=\nu$ is impossible).
In the first case $\nu$ is maximal in~$P_{\cH_\nu,h'}$.
In the second case we have $h'(A)=h_\nu(A)=b$ and therefore $\nu \prec_{h'} b$.
As $\nu=j$ and~$b$ are at neighboring positions in~$\pi'$, we see that $b=\pi_{h_\nu}(i+d+1)$ covers~$\nu$ in~$P_{\cH_\nu,h'}$.
Combining these observations shows that the jump in~$\pi_{h_\nu}$ is minimal, we have $\pi' \in \Pi_{\cH_\nu}$ and the acyclic orientation~$h'$ obtained from $h_\nu$ by the pair flip~$(j,\pi_{h_\nu}(i+1))$ satisfies $\pi'=\pi_{h'}$.

{\bf Case 2:} $j<\nu$.
Note that $p(\pi_{h_\nu})$ and $p(\pi')$ differ in a right jump of the value~$j$.
As the jump is clean by assumption, we have that $\pi_{h_\nu}^{-1}(\nu)=\pi'^{-1}(\nu) \in \{1,\nu\}$.
In fact, we may assume that $\pi_{h_\nu}^{-1}(\nu)=\pi'^{-1}(\nu)=\nu$, as the other case is analogous.
By induction, the jump in $p(\pi_{h_\nu})$ is minimal, we have $p(\pi') \in \Pi_{\cH_{\nu-1}}$ and the acyclic orientation $h':\cE_{\nu-1} \to [\nu-1]$ of~$\cH_{\nu-1}$ obtained from~$h_{\nu-1}$ by the pair flip $(j,\pi_{h_{\nu-1}}(i+1))=(j,\pi_{h_\nu}(i+1))$ satisfies $p(\pi')=\pi_{h'}\in \Pi_{\cH_{\nu-1}}$.
It follows that the jump in~$\pi_{h_\nu}$ is also minimal.

Moreover, we define
\begin{align*}
\hat h(A) & :=\begin{cases}
 \nu & \text{if $\nu \in A$}, \\
 h'(A) & \text{otherwise,} \\
\end{cases} \\
& =\begin{cases}
 \nu & \text{if $\nu \in A$}, \\
 j & \text{if $\nu \notin A$, $h_{\nu-1}(A)=\pi_{h_\nu}(i+1)$ and $j \in A$,} \\
 h_{\nu-1}(A) & \text{otherwise,} \\
\end{cases} \\
& =\begin{cases}
 j & \text{if $h_{\nu}(A)=\pi_{h_\nu}(i+1)$ and $j \in A$,} \\
 h_\nu(A) & \text{otherwise,} \\
\end{cases}
\end{align*}
From the last line of this equation we see that $\hat h$ is obtained from~$h_\nu$ by the pair flip~$(j,\pi_{h_\nu}(i+1))$.
Furthermore, as $\pi'=c_\nu(\pi_{h'})$ we conclude that $\pi'=\pi_{\hat h} \in \Pi_{\cH_\nu}$.
\end{proof}

\begin{proof}[Proof of Lemma~\ref{lem:algJ-flip}]
Let $h,h'\in\AO_\cH$, $j=\pi_h(i)$ and~$d$ be such that $\pi_h$ and $\pi_{h'}$ differ in a minimal jump of the value~$j$ by $d$ steps.
By Lemma~\ref{lem:clean-jumps} we can assume that the jump is clean.
Furthermore, we assume that the jump is a right jump, as the other case follows the same ideas.
By Lemma~\ref{lem:jump-flip} applying the pair flip~$(j,\pi_h(i+1))$ to~$h$ yields an acyclic orientation~$h''\in\AO_\cH$ such that $\pi_{h''}$ is obtained from~$\pi_h$ by a minimal right jump of~$j$.
It follows that $h''=h'$, and the lemma is proved.
\end{proof}

Combining Theorem~\ref{thm:jump}, Lemma~\ref{lem:heo2perm}, and Lemma~\ref{lem:algJ-flip} yields our first main result.

\begin{theorem}
\label{thm:mainhyper}
For every hypergraph $\cH = ([n],\cE)$ in hyperfect elimination order, Algorithm~J with input~$\Pi_\cH$ as defined in~\eqref{eq:PiH} generates a sequence of permutations $\pi_{h_1}, \pi_{h_2}, \ldots$, where $h_1,h_2,\ldots \in \AO_\cH$ such that $h_1,h_2,\ldots $ is a Hamilton path in the flip graph on acyclic orientations of~$\cH$, or equivalently, on the skeleton of the hypergraphic polytope~$Z(\cH)$.
\end{theorem}

\subsection{Application to building sets and nestohedra}
\label{sec:bs}

In what follows, we use the terminology of Postnikov~\cite{MR2487491}, although the notion of building set goes back to De Concini and Procesi~\cite{MR1366622}.

\subsubsection{Acyclic orientations of building sets and nestohedra}

A hypergraph $\cB=(V,\cE)$ is a \emph{building set} if
\begin{enumerate}[label=(\roman*),leftmargin=8mm, noitemsep, topsep=1pt plus 1pt]
\item $\cE$ contains all singletons $\{v\}$ for $v\in V$;
\item for every pair $A,B\in\cE$ with $A\cap B\neq \emptyset$ we have $A\cup B\in\cE$.
\end{enumerate}
A building set is \emph{connected} if $V\in\cE$.

A remarkable property of acyclic orientations of building sets is that the transitive reduction of the corresponding posets are forests.

\begin{lemma}[{\cite[Thm.~7.4]{MR2487491}, \cite[Prop.~3.3]{MR3960512}}]
If $\cB$ is a building set, then
\begin{enumerate}[label=(\roman*),leftmargin=8mm, noitemsep, topsep=1pt plus 1pt]
\item the hypergraphic polytope $Z(\cB)$ of $\cB$ is simple;
\item the flip graph on $\AO_{\cB}$ is regular;
\item for every acyclic orientation $h\in\AO_{\cB}$, the poset $P_{\cH,h}$ is a forest.
\end{enumerate}
\end{lemma}

The hypergraphic polytope~$Z(\cB)$ of a building set~$\cB$ is called a \emph{nestohedron}~\cite{MR2487491}.
The forests corresponding to acyclic orientations of $\AO_{\cB}$ are called $\cB$-forests (or $\cB$-trees if $\cB$ is connected).

\subsubsection{Chordal building sets}

We follow Postnikov, Reiner, and Williams~\cite{MR2520477}\footnote{The original definition is actually reversed, but we need this one for consistency with our definition of perfect elimination order in chordal graphs.} and say that a building set $\cB=([n],\cE)$ is a \emph{chordal building set} if and only if for any $A=\{i_1<i_2<\cdots <i_r\}\in\cE$ and $s\in [r]$, we have $\{i_1<i_2<\cdots <i_s\}\in\cE$.

The point of considering chordal building sets is the following direct connection between chordality and hyperfect elimination orders.

\begin{lemma}
\label{lem:cbs-heo}
A building set is chordal if and only if it is in hyperfect elimination order.
\end{lemma}

\begin{proof}
We first observe that for building sets, condition~\eqref{itm:heo1} for the hyperfect elimination order reduces to the case where $A=B$.
Indeed, if $A,B\in\cE$ both contain~$n$, then $A\cup B\in\cE$.

$(\Rightarrow)$
Consider a chordal building set $\cB=([n],\cE)$.
We need to verify conditions~\eqref{itm:heo1} and \eqref{itm:heo2} of the hyperfect elimination order.
Condition~\eqref{itm:heo1} is satisfied, since for any $A\in\cE$ such that $n\in A$, we have $A- n\in\cE$ from the chordality of $\cB$, and $A-n$ fulfills the requirements of~$X$ in condition~\eqref{itm:heo1}.
It remains to observe that if $\cB$ is chordal, then so is~$\cB_{n-1}$.

$(\Leftarrow)$
Consider a building set $\cB=([n],\cE)$ in hyperfect elimination order.
We only need to prove that condition~\eqref{itm:heo1} implies that for any $A\in\cE$ with $n\in A$, we have $A-n\in\cE$.
Suppose for the sake of contradiction that every hyperedge $X\in\cE$ contained in $A-n$ has size strictly less than~$|A-n|$.
Consider such a hyperedge~$X$ of maximum size, a vertex $a\in X$, and another vertex $b\in (A-n)\setminus X$.
Applying condition~\eqref{itm:heo1} on~$a$ and~$b$ yields another hyperedge~$X'\in\cE$ contained in~$A-n$ that contains both~$a$ and~$b$.
Since $\cH$ is a building set, we must have $X\cup X'\in\cE$, and this set is larger than~$X$, a contradiction.
Therefore, we have $A-n\in\cE$.
The rest follows by induction on~$\cB_{n-1}$.
\end{proof}

Nestohedra of chordal building sets are called \emph{chordal nestohedra}~\cite{MR2520477}.
By combining Lemma~\ref{lem:cbs-heo} and Theorem~\ref{thm:mainhyper}, we therefore obtain new Hamilton paths on chordal nestohedra.

\begin{theorem}
\label{thm:mainbuild}
For every chordal building set $\cB = ([n],\cE)$, Algorithm~J with input~$\Pi_{\cB}$ as defined in~\eqref{eq:PiH} generates a sequence of permutations $\pi_{h_1}, \pi_{h_2}, \ldots$, where $h_1,h_2,\ldots \in \AO_{\cB}$ such that $h_1,h_2,\ldots $ is a Hamilton path in the flip graph on acyclic orientations of~$\cB$, or, equivalently, on the skeleton of the chordal nestohedron~$Z(\cB)$.
\end{theorem}

\subsubsection{Graphical building sets and elimination trees}

For a graph $G=(V,E)$, we let $\cB(G):=(V,\cE)$ be the hypergraph such that
\begin{equation*}
\cE:=\{U\seq V\mid \text{$G[U]$ is connected}\}.
\end{equation*}
Clearly, $\cE$ contains all singletons $\{v\}$, $v\in V$.
Also, if two subsets induce a connected subgraph of~$G$ and have a nonempty intersection, then their union also induces a connected subgraph.
Therefore, $\cB(G)$ is a building set, called the \emph{graphical building set} of~$G$.

An \emph{elimination tree} of a connected graph~$G=(V,E)$ is an unordered rooted tree~$T$ obtained by removing a vertex~$v$ of~$G$ which becomes the root of~$T$, and by recursing on the connected components of $G-v$, whose elimination trees become the subtrees of~$v$ in~$T$.
An \emph{elimination forest} of a (not necessarily connected) graph~$G$ is a set of elimination trees, one for each connected component of~$G$.
An elimination forest of a graph~$G=([n],E)$ can be produced from a permutation on~$[n]$, by choosing to remove at each step the vertex that has the leftmost position in the permutation.
Two elimination forests differ by a \emph{rotation} if there exist two permutations producing them that differ by a single adjacent transposition.
We refer to Figure~\ref{fig:hyper2} for an example of rotation and to \cite{DBLP:conf/soda/CardinalMM22} for more details.

The \emph{graph associahedron} of~$G$ is the hypergraphic polytope~$Z(\cB(G))$.
The vertices of $Z(\cB(G))$ are in one-to-one correspondence with the elimination forests of~$G$, and its skeleton is the \emph{rotation graph} on the elimination forests of~$G$; see \cite{MR2239078,MR2487491,MR2520477,MR3383157}.
Combining this characterization with~\eqref{eq:hypervert} and Lemma~\ref{lem:hyperflip}, we obtain the following.

\begin{lemma}
\label{lem:BG-elim}
Acyclic orientations of a graphical building set~$\cB(G)$ are in one-to-one correspondence with the elimination forests of~$G$, and the flip graph on the acyclic orientations of~$\cB(G)$ is isomorphic to the rotation graph on the elimination forests of $G$.
\end{lemma}

This statement is illustrated in Figure~\ref{fig:hyper2}, showing an example of a rotation between two elimination trees of a graph, and the corresponding pair flip in the acyclic orientations of the building set.

\begin{lemma}[\cite{MR2520477}]
\label{lem:cgbs}
A graphical building set $\cB(G)$ is chordal if and only if $G$ is in perfect elimination order.
\end{lemma}

Note that not all chordal building sets are graphical.
For instance, the hypergraph $([n], \{[i]\mid i\in n\})$ is a chordal building set but is not graphical.
(The corresponding nestohedron is called the \emph{Stanley-Pitman polytope}~\cite{MR1902680}.)

Combining Lemma~\ref{lem:cgbs} and Lemma~\ref{lem:cbs-heo}, we obtain the following.

\begin{corollary}
A graphical building set $\cB(G)$ is in hyperfect elimination order if and only if $G$ is in perfect elimination order.
\end{corollary}

This confirms our earlier claim that our Gray codes for acyclic orientations of hypergraphs in hyperfect elimination order, as given in Theorem~\ref{thm:mainhyper}, generalize the known Gray codes for elimination forests on chordal graphs described recently by Cardinal, Merino, and M\"utze~\cite{DBLP:conf/soda/CardinalMM22}.

\section{Acyclic reorientation lattices}
\label{sec:arl}

In this section, we consider order-theoretic properties of the set of all acyclic orientations of a graph and the existence of Hamilton paths and cycles in lattices defined from acyclic orientations.
In particular, we address a question due to Pilaud~\cite[Problem~51]{pilaud_2022} stated as Problem~\ref{prob:pilaud} in Section~\ref{sec:res2}.
In order to make sense of this question, we first define classes of acyclic digraphs and the lattices that are defined from their acyclic reorientations.

\subsection{Acyclic reorientation lattices of families of acyclic digraphs}

Given an acyclic digraph~$D$, an \emph{acyclic reorientation} of $D$ is an acyclic digraph obtained by flipping the orientation of some arcs of $D$.
The set of all acyclic reorientations of $D$ can be ordered by containment of the sets of arcs that are flipped w.r.t.~$D$; see Figure~\ref{fig:cong2}~(a).
We denote this poset by~$\AR_D$.
Pilaud~\cite{pilaud_2022} gave necessary and sufficient conditions on $D$ for $\AR_D$ to be a lattice.
When those conditions are satisfied, we refer to $\AR_D$ as the \emph{acyclic reorientation lattice} of $D$.

\subsubsection{Acyclic reorientation lattices of vertebrate digraphs}

We use the term \emph{oriented tree} to refer to a digraph whose underlying undirected graph is a tree, and \emph{oriented forest} to refer to a collection of disjoint oriented trees.
Following Pilaud~\cite{pilaud_2022}, we say that a digraph~$D$ is \emph{vertebrate} if the transitive reduction of every induced subgraph of~$D$ is an oriented forest.
It is easy to see that any vertebrate digraph must be acyclic; see~\eqref{eq:inclusion}.

\begin{theorem}[{\cite[Thm.~1]{pilaud_2022}}]
The poset~$\AR_D$ of acyclic reorientations of a digraph~$D$ is a lattice if and only if $D$ is vertebrate.
\end{theorem}

\subsubsection{Acyclic reorientation lattices of peo-consistent digraphs}

A~digraph $D$ is \emph{peo-consistent} if its vertices can be labeled $1,\ldots,n$, such that either $n=1$ or the following three conditions are satisfied:
\begin{enumerate}[label=(\roman*),leftmargin=8mm, noitemsep, topsep=1pt plus 1pt]
\item the vertex~$n$ is a source or a sink of~$D$,
\item the vertex~$n$ is simplicial in the underlying undirected graph of $D$,
\item the digraph~$D-n$ is peo-consistent.
\end{enumerate}

In the following, we often denote a peo-consistent digraph~$D$ with the corresponding labeling of its vertices by~$1,\ldots,n$ as $D=([n],A)$.
Hence a peo-consistent digraph is an orientation of a chordal graph obtained by iteratively adding each vertex in perfect elimination order as a source or a sink in the digraph induced by its predecessors in the order.
Since the underlying undirected graph of a peo-consistent digraph~$D$ can be any graph in perfect elimination order, it has the unique parent-child property by Lemma~\ref{lem:upc-peo}.
The peo-consistent digraphs form a natural class of vertebrate digraphs; see~\eqref{eq:inclusion} and Figure~\ref{fig:classes}~(d).

\begin{lemma}
\label{lem:consistent}
Every peo-consistent digraph is vertebrate.
\end{lemma}

The converse of Lemma~\ref{lem:consistent} does not hold, i.e., not every vertebrate digraph is peo-consistent; see Figure~\ref{fig:classes}~(c).

\begin{proof}
We proceed by induction on the number of vertices, and suppose that the statement holds for all graphs on less than $n$ vertices.
Observe that if $D=([n],A)$ is a peo-consistent digraph, then so is any of its induced subgraph.
Hence from the induction hypothesis, we only need to prove that the transitive reduction~$T_D$ of~$D$ is an oriented tree.
Consider the transitive reduction~$T_{D-n}$ of~$D-n$.
We claim that $T_D$ is obtained by adding the vertex~$n$ as a new leaf to $T_{D-n}$, with either a single incoming arc or a single outgoing arc.
Indeed, since the vertex~$n$ is simplicial, its neighbors must form a directed path in~$T_{D-n}$.
And since $n$ is either added to~$D$ as a source or a sink, there is a single new arc in the transitive reduction, which connects it to either its first or last neighbor in the directed path, respectively.
Therefore, $T_D$ is an oriented tree, as required.
\end{proof}

\subsubsection{Acyclic reorientation lattices of skeletal digraphs}

We say that a digraph~$D$ is \emph{filled} if for any directed path $v_1\rightarrow \cdots\rightarrow v_k$ in~$D$, if the arc~$v_1\rightarrow v_k$ belongs to~$D$, then all arcs $v_i\rightarrow v_j$, $1\leq i<j\leq k$, also belong to~$D$.
A digraph is called \emph{skeletal} if it is both vertebrate and filled.
An alternative definition of a skeletal digraph~$D$ is that $D$ is obtained from an oriented forest by replacing some directed paths by acyclic cliques.
The motivation for introducing skeletal digraphs is the following result due to Pilaud.

\begin{theorem}[{\cite[Thm.~3]{pilaud_2022}}]
\label{lem:semid}
The poset~$\AR_D$ of acyclic reorientations of an acyclic digraph~$D$ is a semidistributive lattice if and only if $D$ is skeletal.
\end{theorem}

We now prove that the class of skeletal digraphs refine that of peo-consistent digraphs; see~\eqref{eq:inclusion}.

\begin{lemma}
\label{lem:skeletal}
Every skeletal digraph is peo-consistent.
\end{lemma}

The converse of Lemma~\ref{lem:skeletal} does not hold, i.e., not every peo-consistent digraph is skeletal; see Figure~\ref{fig:classes}~(b).

\begin{proof}
It suffices to show that in every skeletal digraph~$D$, one can find a source or sink that is simplicial in the underlying undirected graph, i.e., whose neighborhood in~$D$ is a clique.
We refer to such a vertex as a \emph{terminal} vertex.
If a terminal vertex~$v$ exists, then we can remove it and iterate the same argument on the remaining digraph~$D-v$, which is also skeletal.
In fact, we proceed to prove that every skeletal digraph~$D$ has at least two terminal vertices.

We consider a connected skeletal digraph~$D$ and its transitive reduction~$T:=T_D$, which is an oriented tree.
We first argue that we can assume without loss of generality that every source or sink has degree exactly~1.
Indeed, suppose that this is not the case.
Then we partition the arc set of~$T$ into a collection~$\cT$ of subtrees, such that in every subtree, every source or sink has degree exactly~1, and moreover every arc of~$D$ joins two vertices from the same subtree in~$\cT$.
Specifically, for each source~$v$ of out-degree~$d>1$ in~$T$, we split the arc set of~$T$ into $d$ subtrees connected to~$v$, assigning all arcs in the same connected component of~$T-v$ together with the arc that connects this component to~$v$ to the same subtree; see Figure~\ref{fig:skeletal}~(a).
We repeat this for every source of out-degree $>1$, and we proceed similarly for each sink of in-degree~$>1$.
Suppose that we have $|\cT|=k$ trees after this partitioning stage.
If we can find two terminal vertices in~$D[T']$ for every $T'\in\cT$, i.e., in the subgraph of~$D$ induced by the subtree~$T'$, then at least $2k-(2k-1)=2$ of those vertices will also be terminal vertices in~$D$.
This is because $D$ is obtained from gluing together the graphs~$D[T']$, $T'\in\cT$, in a tree-like fashion.
So for the rest of the proof we assume without loss of generality that sources and sinks of~$D$ have degree exactly~1.

Let $S_0$ be the set of sources and $S_1$ the set of sinks of~$D$.
We denote by $M$ the set of all vertices of~$D$ that are neither in~$S_0$ nor in~$S_1$; see Figure~\ref{fig:skeletal}~(b).
In other words, every vertex in~$M$ has at least one in-neighbor and at least one out-neighbor.
We observe that $v\in S_0\cup S_1$ is a terminal vertex if its neighborhood in~$D$ induces a directed path in~$T$.
Indeed, since $D$ is filled, $v$ must simplicial in the underlying undirected graph, and hence $v$ is a terminal vertex.

\begin{figure}
\centering
\makebox[0cm]{ 
\includegraphics{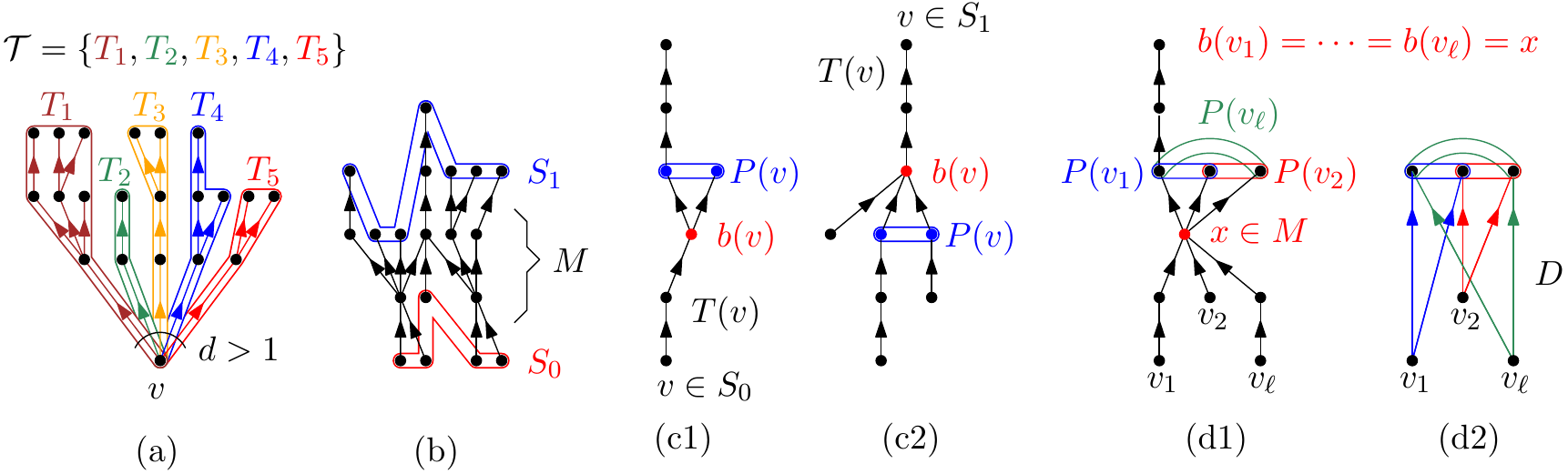}
}
\caption{Illustration of the proof of Lemma~\ref{lem:skeletal}.}
\label{fig:skeletal}
\end{figure}

On the other hand, consider a vertex $v\in S_0$ that is not a terminal vertex; see Figure~\ref{fig:skeletal}~(c1).
Then the neighbors of~$v$ in~$D$ induce an oriented subtree~$T(v)$ of~$T$ rooted at~$v$, all arcs of which are oriented away from~$v$.
As $v$ is not a terminal vertex, $T(v)$ is not a single directed path, but it has a \emph{branch vertex} $b(v)\in M$ with out-degree at least~2 in~$T(v)$.
We denote by $P(v)$ one pair of out-neighbors of~$b(v)$ in~$T(v)$ (pick two arbitrarily if there are more than two), which by definition are also out-neighbors of~$v$ in~$D$.

Consider a vertex $x\in M$ and all sources $v_1,v_2,\ldots ,v_\ell\in S_0$ that have $x$ as their common branch vertex, i.e., that satisfy $b(v_1)=b(v_2)=\cdots =b(v_{\ell})=x$; see Figure~\ref{fig:skeletal}~(d1).
We observe that we must have $\ell\leq d^+(x)-1$.
Indeed, if $\ell>d^+(x)-1$, then the transitive reduction of the subgraph induced by $v_1,v_2,\ldots,v_{\ell}$ and the vertices in $P(v_1),P(v_2),\ldots,P(v_{\ell})$ in $D$ contains a cycle, no arc of which is transitive, contradicting the fact that $D$ is vertebrate; see Figure~\ref{fig:skeletal}~(d2).

Symmetrically, we can consider a vertex $v\in S_1$ that is not a terminal vertex, and define a corresponding branch vertex $b(v)\in M$; see Figure~\ref{fig:skeletal}~(c2).
Then similarly, for a vertex $x\in M$, at most $d^-(x)-1$ vertices from~$S_1$ can have $x$ as their common branch vertex.

We now apply a counting argument.
We have
\begin{equation*}
|S_0|=1 + \sum_{x\in M} (d^-(x) - 1) \quad \text{and} \quad
|S_1|=1 + \sum_{x\in M} (d^+(x) - 1).
\end{equation*}
Suppose without loss of generality that $|S_0|\geq |S_1|$, and first consider the case $|S_0|=|S_1|$.
From the previous observations and the equalities above, we can have at most $|S_1|-1=|S_0|-1$ branch vertices~$b(v)$ for the vertices $v\in S_0$, so one such vertex must be a terminal vertex.
Similarly, we can have at most $|S_0|-1=|S_1|-1$ branch vertices $b(v)$ for the vertices $v\in S_1$, and one such vertex must also be a terminal vertex.
We therefore obtain two terminal vertices.
On the other hand, if $|S_0|>|S_1|$, then we can have at most $|S_1|-1\leq |S_0|-2$ branch vertices~$b(v)$ for the vertices $v\in S_0$, and at least two of them must be terminal vertices.
Hence in all cases, we obtain two terminal vertices, as claimed.
\end{proof}

In fact, there exist chordal graphs, no orientation of which is skeletal.
An example is the \emph{complete $k$-sun}, defined as the graph that is obtained from a $2k$-cycle on the vertices $v_1,\ldots,v_{2k}$ by adding an additional edge between every pair of vertices~$v_{2i}$ and~$v_{2j}$, for every $i\ne j\in [k]$.
The transitive reduction of the complete graph on $v_2,v_4,\ldots,v_{2k}$ is a path, and one can argue that one of the vertices~$v_{2i+1}$ joins two vertices in distance strictly larger than one along that path.
Then this vertex~$v_{2i+1}$ cannot have out-degree~2 or in-degree~2, as this would violate the filled property, and it cannot have out-degree and in-degree~1, as this would violate the vertebrate property.
As a consequence of this observation, if $D$ is skeletal then its underlying undirected graph does not contain a complete $k$-sun as induced subgraph.
Farber~\cite{MR685625} proved that a chordal graph is \emph{strongly chordal} if and only if it contains no induced complete $k$-sun (called trampoline there).
Therefore, if $D$ is skeletal, then its underlying undirected graph is strongly chordal.

\subsection{Quotients of acyclic reorientation lattices}

\subsubsection{Lattice congruences}

Recall the terminology introduced in Section~\ref{sec:prelim}.
We consider an equivalence relation~$\equiv$ on elements of a lattice~$L$.
For an element $x\in L$, we denote by $[x]_{\equiv}$ the equivalence class of $\equiv$ to which $x$ belongs.
An equivalence relation $\equiv$ on $L$ is a \emph{lattice congruence} if it respects joins and meets, i.e., if
\begin{equation*}
(x\equiv x' \ \text{and}\ y\equiv y') \Longrightarrow (x\vee y\equiv x'\vee y'\ \ \text{and}\ x\wedge y\equiv x'\wedge y').
\end{equation*}
For a lattice congruence $\equiv$, the \emph{lattice quotient} $L/{\equiv}$ is the poset on the set of the equivalence classes of~$\equiv$, where for any two equivalence classes $X,Y$, we have $X < Y$ if and only if there is an~$x\in X$ and a~$y\in Y$ such that~$x<y$ in~$L$.
We will need the following well-known lemma, which is a direct consequence of the definition of lattice congruences.

\begin{lemma}
\label{lem:interval}
For any lattice congruence~$\equiv$ of a finite lattice~$L$, and any element $x\in L$, the equivalence class $[x]_{\equiv}$ is an interval of~$L$.
\end{lemma}

The definition of a lattice congruence gives rise to so-called \emph{forcing rules}.
These rules state that if some pair of elements of a lattice are congruent, then so must be some other pairs.
We now state two forcing rules that we need in the following; see Figure~\ref{fig:forcing}.
A \emph{diamond} is a four-element poset $(\{a,b,c,d\},<)$ with $a<b<d$ and $a<c<d$ and no other cover relations.
A \emph{hexagon} is a six-element poset $(\{a,b,c,d,e,f\},<)$ with $a<b<d<f$ and $a<c<e<f$ and no other cover relations.

\begin{figure}[h!]
\includegraphics{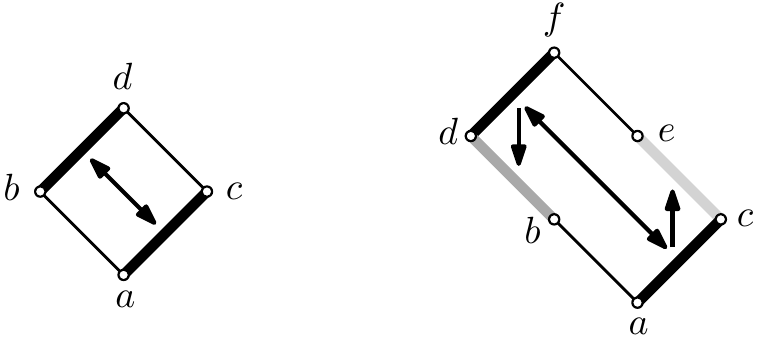}
\caption{Forcing rules in a lattice congruence: diamond rule (left) and hexagon rule (right).
Edges indicate cover relations, and bold edges indicate that the two elements belong to the same equivalence class.
The arrows indicate implications.}
\label{fig:forcing}
\end{figure}

\begin{lemma}
\label{lem:forcing}
Let $\equiv$ be a congruence of a lattice~$L$.
\begin{description}
\item[Diamond rule] For every diamond sublattice $\{a,b,c,d\}$ of~$L$ with $a<b<d$ and $a<c<d$, we have $a\equiv c\Leftrightarrow b\equiv d$.
\item[Hexagon rule] For every hexagon sublattice $\{a,b,c,d,e,f\}$ of~$L$ with $a<b<d<f$ and $a<c<e<f$, we have $a\equiv c\Leftrightarrow d\equiv f$ and $(a\equiv c \text{ and } d\equiv f)\Rightarrow (b\equiv d \text{ and } c\equiv e)$.
\end{description}
\end{lemma}

\begin{proof}
The statements are derived by elementary applications of the definition of lattice congruences.
For diamonds, we have $a\equiv c\Rightarrow a\vee b\equiv c\vee b \Rightarrow b\equiv d$.
Symmetrically, we have $b\equiv d\Rightarrow b\wedge c\equiv d\wedge c\Rightarrow a\equiv c$.

For hexagons, we again apply the definition of a lattice congruence as follows: $a\equiv c\Rightarrow a\vee b \equiv c\vee b\Rightarrow b\equiv f$.
As equivalence classes are intervals, we also have $b\equiv d$ and $d\equiv f$.
Symmetrically, from $d\equiv f$ we obtain $d\wedge e\equiv f\wedge e\Rightarrow a\equiv e$ and hence $a\equiv c$ and~$c\equiv e$.
\end{proof}

In what follows, we consider quotients $\AR_D/{\equiv}$ of the acyclic reorientation lattice~$\AR_D$ of an acyclic digraph~$D$.

We emphasize that the diamond and hexagon rule stated in Lemma~\ref{lem:forcing} are necessary, but may not be sufficient to completely define the forcing relations for congruences of~$\AR_D$.
It is known that such local forcing rules are sufficient whenever the lattice $\AR_D$ is \emph{polygonal}~\cite[Thm.~9-6.5]{MR3645055}, which is the case if and only if it is semidistributive~\cite[Thm.~9-3.8 and 9-6.10]{MR3645055}.
Hence from Lemma~\ref{lem:semid}, $\AR_D$ is polygonal if and only if $D$ is skeletal, and in that case the two rules in Lemma~\ref{lem:forcing} completely characterize forcing relations in congruences of~$\AR_D$.
Note that starting from Section~\ref{sec:quotient-algo}, we will only assume that $D$ is peo-consistent, and not necessarily skeletal, so the diamond and hexagon rule are necessary, but may not be sufficient to characterize all forcing relations in a congruence (but our proof works regardless).

\subsubsection{Quotientopes}

Given a skeletal digraph~$D$, Pilaud~\cite{pilaud_2022} realizes any lattice quotient of~$\AR_D$ as a polytope, called a \emph{quotientope}.
Previously, the notion of quotientope has been used for the polytopal realization of lattice quotients of the weak order on permutations in the same manner~\cite{MR3964495,MR4311892}.
Recall that the weak order on permutations is~$\AR_D$ when $D$ is an acyclic complete graph.
The cover graph of a lattice quotient is exactly the skeleton of the corresponding quotientope.
Therefore, Problem~\ref{prob:pilaud} is equivalent to asking whether the skeleta of these quotientopes admit a Hamilton cycle.

\subsection{Algorithm}
\label{sec:quotient-algo}

We now give an algorithm to construct Hamilton paths in the cover graphs of quotients of acyclic reorientation lattices of peo-consistent digraphs.

\subsubsection{Restrictions, rails, ladders, and projections}

We introduce some notations that will be useful for inductive reasonings.
For the rest of this paper let $D=([n],A)$ be a peo-consistent digraph.
For an acyclic reorientation~$E\in\AR_{D_{n-1}}$, we denote by~$c(E)$ the acyclic reorientation of~$D\in\AR_D$ obtained from~$E$ by adding the last vertex~$n$ as a source or as a sink, according to how it appears in~$D$.
Similarly, we write~$\bar{c}(E)$ for the acyclic reorientation of~$D$ obtained from~$E$ by adding the last vertex~$n$ as a source if it is a sink in~$D$, or as a sink if it is a source in~$D$, hence oppositely to how it appears in~$D$.
Given a lattice congruence~$\equiv$ on~$\AR_D$, we define the \emph{restriction} $\equiv^*$ as the relation on~$\AR_{D_{n-1}}$ induced by all acyclic reorientations of~$D$ in which no arc incident with~$n$ is reoriented with respect to~$D$, i.e., for any two acyclic reorientations~$E$ and~$F$ in~$\AR_{D_{n-1}}$ we have $E \equiv^* F\Leftrightarrow c(E)\equiv c(F)$.

\begin{lemma}
For every lattice congruence~$\equiv$ of~$\AR_D$, the restriction~$\equiv^*$ is a lattice congruence of~$\AR_{D_{n-1}}$.
\end{lemma}

\begin{proof}
This follows straightforwardly from the definition of lattice congruence and restriction and the observation that for any two acyclic reorientations~$E$ and~$F$ in~$\AR_{D_{n-1}}$, we have $c(E) \vee c(F) = c(E \vee F)$ and $c(E) \wedge c(F) = c(E \wedge F)$.
\end{proof}

A \emph{rail} in the acyclic reorientation lattice~$\AR_D$ is a maximal subposet of~$\AR_D$ induced by all acyclic reorientations of~$D$ that agree on the orientations of all the arcs in~$D_{n-1}$, i.e., that differ only in the orientation of the arcs incident with~$n$.
For a reorientation~$E\in\AR_{D_{n-1}}$, we denote the corresponding rail in~$\AR_D$ by~$r(E)$; see Figure~\ref{fig:ladder}.
The minimum element of the rail~$r(E)$ is~$c(E)$, the maximum element is~$\bar{c}(E)$, and the number of elements on the rail equals the number of arcs incident with~$n$ in~$D$ (i.e., the degree of~$n$).
The number of rails in~$\AR_D$ is equal to~$|\AR_{D_{n-1}}|$, and these rails form a partition of~$\AR_D$ into chains of the same size.

\begin{lemma}
\label{lem:rail}
For every lattice congruence~$\equiv$ of~$\AR_D$, every equivalence class~$X$ of~$\equiv$, and every rail~$r$ of~$\AR_D$, the intersection~$X\cap r$ is an interval in~$r$.
\end{lemma}

\begin{proof}
From Lemma~\ref{lem:interval}, we know that $X$ is an interval of~$\AR_D$.
Since every rail~$r$ is a chain of~$\AR_D$, its intersection with~$X$ must be an interval in~$r$.
\end{proof}

A \emph{ladder} is the subposet of~$\AR_D$ induced by a pair of rails~$r(E)$ and~$r(F)$ for which~$E,F\in\AR_{D_{n-1}}$ differ in a flip of a single arc, i.e., $E$ and~$F$ are in cover relation in~$\AR_{D_{n-1}}$; see Figure~\ref{fig:ladder}.
We denote this ladder by~$\ell(E,F)$.
The cover graph of a ladder consists of two paths belonging to the rails, and of additional cover edges that we call \emph{stairs}.
We will need the following property of ladders, which explains the chosen name.

\begin{figure}[h!]
\centering
\makebox[0cm]{ 
\includegraphics{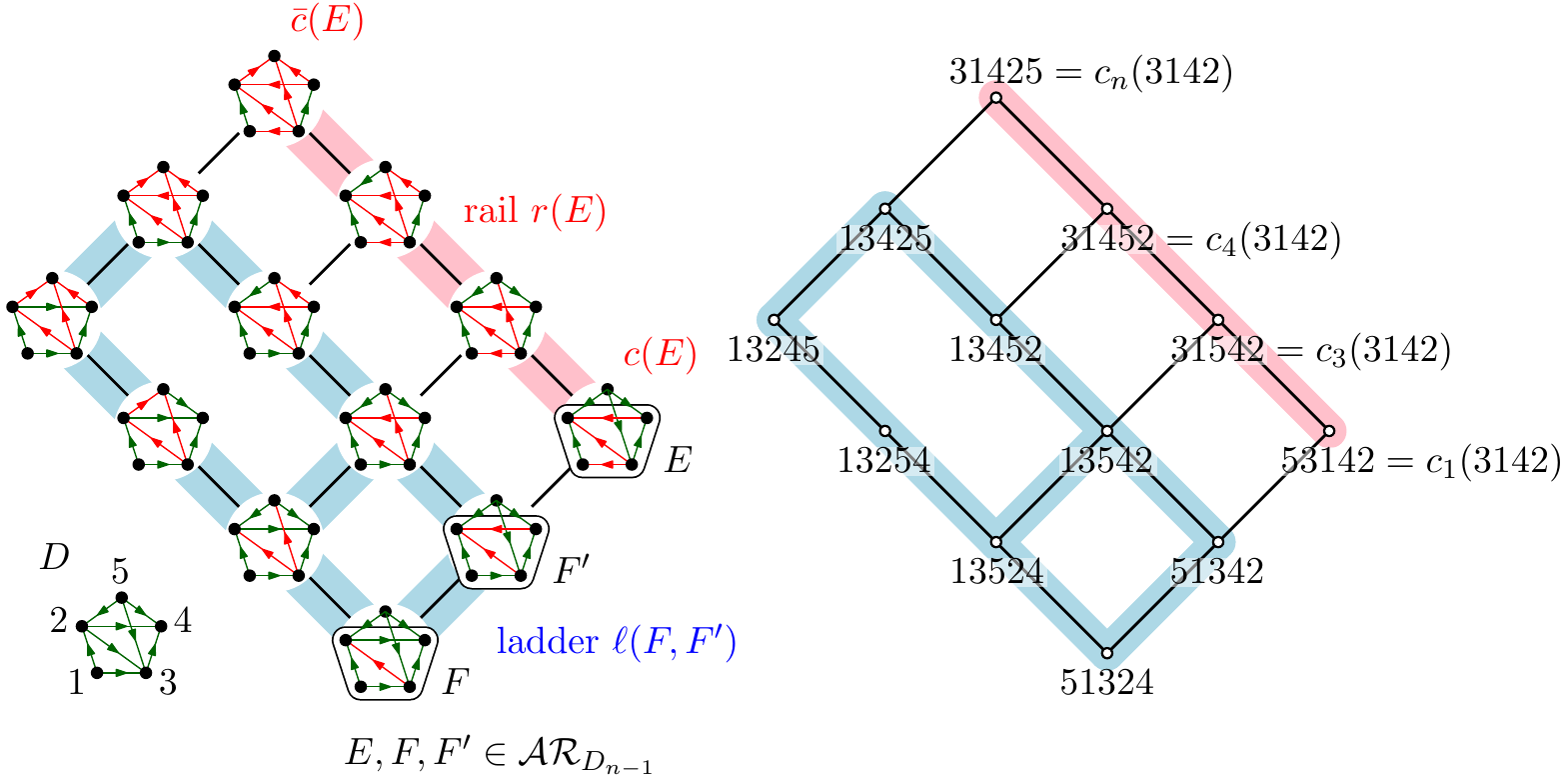}
}
\caption{Illustration of rails and ladders.
A sublattice of~$\AR_D$ is shown on the left (reoriented arcs w.r.t.~$D$ are highlighted), with the encoding of acyclic reorientations by permutations given by Lemma~\ref{lem:peo2perm} shown on the right.
A rail and a ladder and are highlighted.}
\label{fig:ladder}
\end{figure}

\begin{lemma}
\label{lem:ladder}
For any ladder~$\ell(E,F)$, where $E,F\in\AR_{D_{n-1}}$, the first and last pair of reorientations in both rails forms a stair, and the interval between any two successive stairs in the ladder is either a diamond or a hexagon.
\end{lemma}

\begin{proof}
As $E$ and~$F$ differ in a flip of some arc, $c(E)$ and~$c(F)$ are both acyclic reorientations of~$D$ that differ in a flip of the same arc.
Similarly, $\bar{c}(E)$ and~$\bar{c}(F)$ are both acyclic reorientations of~$D$ that differ in a flip of the same arc.
Consequently, $(c(E),c(F))$ is the first stair of the ladder~$\ell(E,F)$, and $(\bar{c}(E),\bar{c}(F))$ is the last stair of the ladder.

Now consider any stair~$(E',F')$ of this ladder and denote by~$a$ the arc of~$D$ that has a distinct orientation in~$E'$ and~$F'$.
Furthermore, let $E''\in r(E)$ and~$F''\in r(F)$ be the reorientations that cover~$E'$ and~$F'$ in their respective rails.
We then consider two cases.

The first case is when~$E''$ and~$F''$ are obtained from~$E'$ and~$F'$, respectively, by flipping the same arc~$b$.
In that case, $E''$ and~$F''$ clearly differ in a flip of the single arc~$a$, and form the next stair~$(E'',F'')$ in the ladder.
Hence the two successive stairs form a diamond.

The second case is when~$E''$ is obtained from~$E'$ by flipping an arc~$b$, and $F''$ is obtained from~$F'$ by flipping another arc~$c\neq b$.
Then the pair~$(E'',F'')$ is not a stair.
Since $E''\in r(E)$ and $F''\in r(F)$, it must be the case that $b$ and~$c$ are both incident to the vertex~$n$.
We assume without loss of generality that the vertex~$n$ is a sink in~$D$, which means that the arcs~$b$ and~$c$ are incoming to~$n$ in~$E'$ and~$F'$.
Since $n$ is a simplicial vertex, the other two endpoints of~$b$ and~$c$ must be adjacent.
And it must be the case that arc~$a$ connects these two endpoints of~$b$ and~$c$, as otherwise one of~$E''$ or~$F''$ would not be acyclic.
So the three arcs $a,b,c$ form a triangle.

We claim that arc~$c$ can be flipped in~$E''$ to obtain the next reorientation~$E'''\in r(E_{n-1})$.
Indeed, suppose for contradiction that flipping $c$ creates a cycle.
Then either this cycle does not use~$a$ and must also be present in~$F''$ (as it cannot use~$b$), or it uses arc~$a$, but then there already is a shorter cycle in~$E''$ that uses~$b$ instead of $c,a$.
This proves the claim.
Since we can flip arc~$c$ in~$E''$, this must yield the next reorientation~$E'''$ on the rail~$r(E)$.

Similarly, the arc~$b$ can be flipped in~$F''$ to obtain the next reorientation~$F'''\in r(F)$.
Now~$E'''$ and~$F'''$ only differ by a flip of the arc~$a$, and hence must form the next stair of the ladder.
In this case, the interval induced by the two successive stairs in the ladder is a hexagon.

By applying this reasoning starting from the first stair on the ladder~$\ell(E,F)$ and ending with the last stair, we obtain that the only stairs are the pairs identified above, which completes the proof.
\end{proof}

From the proof above we see that each ladder has at most one hexagon formed by two consecutive stairs.
More specifically, a ladder~$\ell(E,F)$ for $E,F\in\AR_{D_{n-1}}$ consists of a single hexagon and diamonds if and only if the vertex~$n$ is incident to both endpoints of the arc~$a$ in which $E$ and~$F$ differ, and otherwise the ladder has only diamonds.
For example, in Figure~\ref{fig:ladder}, the ladder~$\ell(F,F')$ is of the first type, and the ladder~$\ell(F',E)$ is of the second type.

Given a set $X\seq\AR_D$, we define the \emph{projection} $p(X) := \{E_{n-1} \mid E\in X \}$.
The following lemma is a crucial ingredient for our algorithm, and it is proved by repeated applications of Lemma~\ref{lem:forcing}.

\begin{lemma}
\label{lem:projection}
For every lattice congruence~$\equiv$ of~$\AR_D$ and every equivalence class~$X$ of~$\equiv$, the projection~$p(X)$ is an equivalence class of~$\equiv^*$.
In particular, for any two equivalence classes~$X,Y$, we either have $p(X)=p(Y)$ or $p(X)\cap p(Y)=\emptyset$.
\end{lemma}

\begin{proof}
For the sake of contradiction suppose that there is an equivalence class~$X$ of~$\equiv$ for which~$p(X)$ is not an equivalence class of~$\equiv^*$.
Pick some acyclic reorientation from~$p(X)\seq\AR_{D_{n-1}}$, and let $Y$ be its equivalence class of~$\equiv^*$.
As $Y\neq p(X)$ we must have $p(X)\setminus Y\neq \emptyset$ or $Y\setminus p(X)\neq \emptyset$.

We first consider the case~$p(X)\setminus Y\neq \emptyset$.
In this case there must be $E,F\in X$ that are in cover relation in~$\AR_D$ with $E\equiv F$ such that $p(E)=E_{n-1}\in p(X)\setminus Y$ and $p(F)=F_{n-1}\in Y$, in particular $E_{n-1}\notin Y$.
This means that in the ladder~$\ell(E_{n-1},F_{n-1})$, the endpoints~$E$ and~$F$ of the stair~$(E,F)$ are congruent.
By using Lemma~\ref{lem:ladder} and repeatedly applying Lemma~\ref{lem:forcing}, it follows that the endpoints of every other stair of the ladder must be congruent as well, in particular the endpoints~$c(E_{n-1})$ and $c(F_{n-1})$ of the first stair.
Consequently, we have $c(E_{n-1})\equiv c(F_{n-1})$ and therefore $E_{n-1}\equiv^* F_{n-1}$, a contradiction to the fact that $E_{n-1}\notin Y$ and $F_{n-1}\in Y$.

We now consider the case~$Y\setminus p(X)\neq \emptyset$.
In this case there must be $E,F\in Y$ that are in cover relation in~$\AR_{D_{n-1}}$ with $E\equiv^* F$ such that $E\in Y\setminus p(X)$ and $F\in p(X)$, in particular $E\notin p(X)$.
This means that in the ladder~$\ell(E,F)$, the endpoints~$E$ and~$F$ of the first stair~$(c(E),c(F))$ are congruent.
By using Lemma~\ref{lem:ladder} and repeatedly applying Lemma~\ref{lem:forcing}, it follows that the endpoints of every other stair of the ladder must be congruent as well.
Consequently, since there is a stair~$(E',F')$ with $F'\in X$ (as $p(F')=F\in p(X)$), we must have $E'\in X$ as well and therefore $p(E')=E\in p(X)$, a contradiction to the fact that~$E\notin p(X)$.
\end{proof}

\begin{lemma}
\label{lem:minmax-rail}
For every lattice congruence $\equiv$ of $\AR_D$, either for every~$E\in\AR_{D_{n-1}}$ there are two distinct equivalence classes of~$\equiv$ containing~$c(E)$ and~$\bar{c}(E)$, or for every rail~$r(E)$, $E\in\AR_{D_{n-1}}$, all reorientations on that rail belong to the same equivalence class.
\end{lemma}

\begin{proof}
By the forcing rules described in Lemma~\ref{lem:forcing}, if the reorientations of some rail $r(E)$, $E\in\AR_{D_{n-1}}$, are pairwise congruent, then so are the reorientations of every rail~$r(F)$, $F\in\AR_{D_{n-1}}$, that forms a ladder~$\ell(E,F)$ with~$r(E)$.
Repeatedly applying this observation shows that in this case all reorientations on every rail are congruent.
Otherwise, every rail intersects with at least two equivalence classes of~$\equiv$.
From Lemma~\ref{lem:interval}, we know that every equivalence class of~$\equiv$ intersects each rail~$r(E)$, $E\in\AR_{D_{n-1}}$, in an interval, hence the minimum~$c(E)$ and maximum~$\bar{c}(E)$ must belong to distinct equivalence classes, as claimed.
\end{proof}

\subsubsection{Selection of representatives and encoding as a zigzag language}

As mentioned before, $D=([n],A)$ is assumed to be a peo-consistent acyclic digraph throughout.
For any congruence~$\equiv$ of the acyclic reorientation lattice~$\AR_D$, we say that $R\seq \AR_D$ is a set of \emph{representatives} for the equivalence classes~$\AR_D/{\equiv}$ if and only if for every equivalence class~$X\in\AR_D/{\equiv}$, we have $|X\cap R|=1$.

We now define a set of representatives~$R_D$.
If $D$ is empty, then we define $R_D:=\emptyset$.
Otherwise, we consider the set $R_{D_{n-1}}\seq \AR_{D_{n-1}}$ of representatives for the restriction~$\equiv^*$ of~$\equiv$ on~$D_{n-1}$.
For every acyclic reorientation $E\in R_{D_{n-1}}$, we consider the rail~$r(E)$ in~$\AR_D$.
By Lemma~\ref{lem:minmax-rail} there are two possible cases.
If for every~$E\in\AR_{D_{n-1}}$ there are two distinct equivalence classes of~$\equiv$ containing~$c(E)$ and~$\bar{c}(E)$, then we define a set~$R_{r(E)}$ for all $E\in R_{D_{n-1}}$ as follows.
For each equivalence class~$X\in \AR_D/{\equiv}$ such that $X\cap r(E)\neq\emptyset$, we pick exactly one element from~$X\cap r(E)$ to be included in the set~$R_{r(E)}$.
In particular, we always pick~$c(E)$ and~$\bar{c}(E)$.
We then define
\begin{subequations}
\label{eq:RD}
\begin{equation}
\label{eq:RD1}
R_D := \bigcup_{E\in R_{D_{n-1}}} R_{r(E)}.
\end{equation}
Otherwise, by Lemma~\ref{lem:minmax-rail} for every rail~$r(E)$, $E\in \AR_{D_{n-1}}$, all reorientations on that rail belong to the same equivalence class.
For each $E\in R_{D_{n-1}}$ we select for the set~$R_D$ the reorientation that consists of adding the vertex~$n$ to~$E$ as a sink, i.e., we define
\begin{equation}
\label{eq:RD2}
R_D := \begin{cases}
\{ c(E) \mid E\in R_{D_{n-1}} \} & \text{if $n$ is a sink in~$D$,} \\
\{ \bar{c}(E) \mid E\in R_{D_{n-1}} \} & \text{if $n$ is a source in~$D$.}
\end{cases}
\end{equation}
\end{subequations}

In order to apply Algorithm~J, we interpret the representatives for~$\AR_D/{\equiv}$ as a zigzag language of permutations.
Recall that in Lemma~\ref{lem:peo2perm} we defined a map from acyclic orientations of a graph in perfect elimination order to permutations.
We reuse this definition here, and with any reorientation~$E$ of~$D$ (as $D$ is peo-consistent, $1,\ldots,n$ is a perfect elimination ordering of the underlying undirected graph), we associate the permutation~$\pi_E$, and we define
\begin{equation}
\label{eq:PiD}
\Pi_D := \{ \pi_E \mid E\in R_D \}.
\end{equation}

\begin{lemma}
\label{lem:cong}
For every lattice congruence~$\equiv$ of~$\AR_D$, the set~$R_D$ defined in~\eqref{eq:RD} is indeed a set of representatives for~$\AR_D/{\equiv}$.
Furthermore, the set~$\Pi_D$ defined in~\eqref{eq:PiD} is a zigzag language of permutations satisfying condition~(z1) if~\eqref{eq:RD1} holds and condition~(z2) if~\eqref{eq:RD2} holds.
\end{lemma}

\begin{proof}
We argue by induction on~$n$.
The statement trivially holds for~$n=0$.
For the induction step, suppose that $R_{D_{n-1}}$ is a set of representatives for~$\AR_{D_{n-1}}/{\equiv^*}$ and that $\Pi_{D_{n-1}}$ is a zigzag language of permutations.
Observe that $\Pi_{D_{n-1}} = \{ p(\pi_E) \mid \pi_E \in \Pi_D \}$.
Hence, in order to show that $\Pi_D$ is a zigzag language, we only need to show that either condition~(z1) or~(z2) as stated in Section~\ref{sec:greedy} is met.
Consider the two cases~\eqref{eq:RD1} and~\eqref{eq:RD2} in the inductive definition of~$R_D$.

In the first case, by Lemma~\ref{lem:projection} and the induction hypothesis, for each equivalence class~$X$ of~$\AR_D$, the projection~$p(X)$ has exactly one element in common with~$R_{D_{n-1}}$.
Hence, $X$ has a nonempty intersection with exactly one rail~$r(E)$ for some $E\in R_{D_{n-1}}$.
By construction, we then have $|R_{r(E)}\cap X|=1$ and consequently $|R_D\cap X|=1$ by the definition~\eqref{eq:RD1}.
It follows that $R_D$ is a set of representatives for~$\AR_D/{\equiv}$.
Moreover, by construction, for every $E \in R_{D_{n-1}}$, the reorientations~$c(E)$ and~$\bar{c}(E)$ are in~$R_{r(E)}$ and consequently also in~$R_D$.
This implies that $c_1(\pi_E)$ and~$c_n(\pi_E)$ are in~$\Pi_D$, i.e., $\Pi_D$ satisfies condition~(z1).

In the second case, by Lemma~\ref{lem:projection} every equivalence class~$X$ of~$\AR_D$ satisfies $X=\bigcup_{E \in X^*} r(E)$.
By the induction hypothesis, every equivalence class of~$\AR_{D_{n-1}}$ has exactly one element in common with~$R_{D_{n-1}}$.
Therefore, for $R_D$ as defined by~\eqref{eq:RD2}, each equivalence class of~$\AR_D$ has exactly one element in common with~$R_D$.
It follows that $R_D$ is a set of representatives for~$\AR_D/{\equiv}$.
Furthermore, in this case we have $\Pi_D = \{c_n(\pi_E) \mid E \in R_{D_{n-1}} \} = \{c_n(\pi_E) | \pi_E \in \Pi_{D_{n-1}} \}$, i.e., $\Pi_D$ satisfies condition~(z2).
\end{proof}

\subsubsection{Algorithm~J for quotients of acyclic reorientation lattices}

\begin{lemma}
\label{lem:cong-J}
When running Algorithm~J with input~$\Pi_D$ as defined in~\eqref{eq:PiD}, then for any two permutations~$\pi_E, \pi_F$ that are visited consecutively, the corresponding equivalence classes~$[E]_{\equiv}$ and~$[F]_{\equiv}$ form a cover relation in the quotient~$\AR_D/{\equiv}$.
\end{lemma}

\begin{proof}
We proof the lemma by induction.
The statement is vacuously true for~$n=0$.
For the induction step we assume that it holds for the input~$\Pi_{D_{n-1}}$.

By Lemma~\ref{lem:cong}, $\Pi_D$ is a zigzag language of permutations.
Furthermore, if~\eqref{eq:RD1} holds, then it satisfies condition~(z1), so the permutations in~$\Pi_D$ are generated in the sequence~$J(\Pi_D)$ defined in~\eqref{eq:JLn1}.
Observe that all permutations in~$\lvec{c}(\pi_E)$ or~$\rvec{c}(\pi_E)$, where $E\in R_{D_{n-1}}\seq \AR_{D_{n-1}}$, correspond to acyclic reorientations that lie on the rail~$r(E)=:(E_1,\ldots,E_d)=(c(E),\ldots,\bar{c}(E))$.
If $\pi_{E_i},\pi_{E_j}$, where $1\leq i<j\leq d$, are visited consecutively, then by the definition of the representatives~$R_{r(E)}$ there is an integer~$s$ with $i\leq s<j$ such that
\begin{equation*}
E_i\equiv E_{i+1}\equiv \cdots\equiv E_s\not\equiv E_{s+1}\equiv E_{s+2}\equiv \cdots\equiv E_j,
\end{equation*}
so~$[E_i]_\equiv$ and~$[E_j]_\equiv$ form a cover relation in $\AR_D/{\equiv}$.
Moreover, if $\pi_E$ and~$\pi_F$, where $E,F\in R_{D_{n-1}}\seq \AR_{D_{n-1}}$, are visited consecutively in~$J(\Pi_{D_{n-1}})$, then transitioning from the last permutation of~$\lvec{c}(\pi_E)$ to the first permutation of~$\rvec{c}(\pi_F)$, or from the last permutation of~$\rvec{c}(\pi_E)$ to the first permutation of~$\lvec{c}(\pi_F)$, corresponds to moving from~$c(E)$ to~$c(F)$, or from~$\bar{c}(E)$ to~$\bar{c}(F)$.
Consequently, as $[E]_{\equiv^*}$ and~$[F]_{\equiv^*}$ form a cover relation in the quotient~$\AR_{D_{n-1}}/{\equiv^*}$ by induction, we obtain with the help of Lemma~\ref{lem:projection} that $[c(E)]_{\equiv}$ and~$[c(F)]_{\equiv}$, and also $[\bar{c}(E)]_{\equiv}$ and~$[\bar{c}(F)]_{\equiv}$ form a cover relation in the quotient~$\AR_D/{\equiv}$.

On the other hand, if~\eqref{eq:RD2} holds, then the zigzag language~$\Pi_D$ satisfies condition~(z2), so the permutations in~$\Pi_D$ are generated in the sequence~$J(\Pi_D)$ defined in~\eqref{eq:JLn2}.
In this case, the claim follows immediately by induction.
\end{proof}

Combining Theorem~\ref{thm:jump}, Lemma~\ref{lem:cong}, and Lemma~\ref{lem:cong-J} yields our second main theorem.
Note that by Lemma~\ref{lem:skeletal}, Theorem~\ref{thm:mainquotient} below applies in particular to skeletal digraphs~$D$, thus addressing Problem~\ref{prob:pilaud}.

\begin{theorem}
\label{thm:mainquotient}
For every peo-consistent digraph~$D$ and every lattice congruence $\equiv$ of~$\AR_D$, Algorithm~J with input~$\Pi_D$ as defined in~\eqref{eq:PiD} generates a sequence $\pi_{E_1}, \pi_{E_2}, \ldots$ of permutations from~$R_D$, where $E_1,E_2,\ldots \in \AR_D$ such that $[E_1]_{\equiv}, [E_2]_{\equiv},\ldots $ is a Hamilton path in the cover graph of the lattice quotient~$\AR_D/{\equiv}$.
\end{theorem}

By the remarks after Lemma~\ref{lem:peo2perm}, the minimum~$D$ of the acyclic reorientation lattice~$\AR_D$ is encoded by the identity permutation~$\pi_D=\ide_n$, which by Theorem~\ref{thm:jump} can be used for initialization in Algorithm~J, if and only if the vertex~$i$ is a sink in~$D_i$ for all $i\in[n]$.

\bibliographystyle{alpha}
\bibliography{refs}

\end{document}